\documentclass[a4paper, 10pt, twoside, notitlepage]{amsart}
\usepackage{amsmath,amsfonts,amssymb}
\usepackage{stmaryrd}
\usepackage[margin=1.5in,includefoot,footskip=30pt]{geometry}
\usepackage{graphicx}
\usepackage{caption}
\usepackage{todonotes}
\usepackage{marginnote}
\usepackage{subcaption}
\usepackage{float}

\usepackage{tikz}
\usetikzlibrary{patterns}
\usetikzlibrary{decorations.pathreplacing}

\usepackage{mathtools}
\usepackage[ocgcolorlinks, linkcolor=blue, citecolor=magenta]{hyperref}
\usepackage[capitalise, nameinlink, noabbrev]{cleveref}

\newtheorem{lemma}{Lemma}
\newtheorem{proposition}{Proposition}
\newtheorem{theorem}{Theorem}

\newtheorem{corollary}{Corollary}
\newtheorem{remark}{Remark}

\newcommand{\norm}[1]{\left\|#1 \right\|}
\newcommand{\tnorm}[1]{|\!|\!| #1|\!|\!|}
\newcommand{\jump}[1]{\llbracket #1 \rrbracket }
\usepackage[foot]{amsaddr}

\def\R{\mathbb R}

\def\Pe{\operatorname{Pe}}

\begin{document}

\title[Stabilized FEM for ill-posed convection--diffusion problems. II]{A stabilized finite element method for inverse problems subject to the convection--diffusion equation. II: convection-dominated regime}

\author{Erik Burman}
\author{Mihai Nechita}
\author{Lauri Oksanen}

\address{Department of Mathematics, University College London, Gower Street, London UK, WC1E 6BT.}
\email{\{e.burman, mihai.nechita.16, l.oksanen \}@ucl.ac.uk}

\date{\today}

\begin{abstract}
We consider the numerical approximation of the ill-posed data assimilation problem for stationary convection--diffusion equations and extend our previous analysis in [\emph{Numer. Math.} 144, 451--477, 2020] to the convection-dominated regime. Slightly adjusting the stabilized finite element method proposed for dominant diffusion, we draw upon a local error analysis to obtain quasi-optimal convergence along the characteristics of the convective field through the data set. The weight function multiplying the discrete solution is taken to be Lipschitz continuous and a corresponding super approximation result (discrete commutator property) is proven. The effect of data perturbations is included in the analysis and we conclude the paper with some numerical experiments.
\end{abstract}	

\maketitle
\section{Introduction}
\label{section:intro}
In this work, we consider a data assimilation problem for a stationary convection--diffusion equation
\begin{equation}\label{eq:intro_model_problem}
\mathcal{L} u := - \mu \Delta u + \beta \cdot \nabla  u = f \quad \text{in } \Omega\subset \R^n,
\end{equation}
when convection dominates, that is $0 < \mu \ll |\beta|$, and complement the diffusion-dominated case discussed in the first part \cite{BNO20}.
We assume that $\Omega \subset \R^n$ is open, bounded and connected, and there exists a solution $u\in H^2(\Omega)$ to \eqref{eq:intro_model_problem}.
The problem under study is to approximate the solution $u$ given the source $f$ in $\Omega$ and the perturbed restriction $\tilde U_\omega = u\vert_{\omega} + \delta$ of the solution to an open subset $\omega \subset \Omega$. The perturbation $\delta$ is taken in $L^2(\omega)$. Notice that we consider no boundary conditions on $\partial \Omega$. This is a linear ill-posed problem also known as unique continuation.

To start with, let us briefly recall the main results obtained in the first part. Consider an open bounded set $B\subset \Omega$ that contains the data region $\omega$ such that $B\setminus \omega$ does not touch the boundary of $\Omega$. For $u\in H^1(\Omega)$, the following conditional stability estimate was proven for $\mu>0$ and $\beta \in L^{\infty}(\Omega)^n$,
\begin{equation}\label{eq:continuum_stability}
\norm{u}_{L^2(B)} \le C_{st} \left( \norm{u}_{L^2(\omega)} + \tfrac{1}{\mu} \norm{\mathcal{L} u}_{H^{-1}(\Omega)} \right)^\kappa
\norm{u}_{L^2(\Omega)}^{1-\kappa},
\end{equation}
where the H\"older exponent $\kappa \in (0,1)$ depends only on the geometric setting. In the case of simple geometric configurations, e.g. when $\omega,\, B,\, \Omega$ are three concentric balls, the exponent $\kappa \in (0,1)$ can be given explicitly, see \cite[Remark 1]{BNO20}.
The stability constant $C_{st}$ is given explicitly in terms of the physical parameters
\begin{equation}\label{eq:cstab}
C_{st} = C_1 \exp\left(C_2 (1 + \tfrac{|\beta|}{\mu})^2\right), \quad |\beta| := \norm{\beta}_{L^\infty(\Omega)^n},
\end{equation}
with constants $C_1,\, C_2 > 0$ depending only on the geometry.
Note that the continuum estimate \eqref{eq:continuum_stability} is valid in both the diffusion-dominated and convection-dominated regimes, and that the stability constant $C_{st}$ is uniformly bounded when diffusion dominates. However, when convection dominates $C_{st}$ grows exponentially, rendering the stability estimate ineffective in practice.
We also recall that for global unique continuation from $\omega$ to the entire $\Omega$ the stability is no longer H\"older, but logarithmic, that is the modulus of continuity for the given data is not $|\cdot|^\kappa$ any more, but $|\log(\cdot)|^{-\kappa}$.

On the discrete level, the continuum estimate \eqref{eq:continuum_stability} was combined with a stabilized linear finite element method to obtain convergence orders for the approximate solution.
More precisely, for a mesh size $h$, and defining the P\'eclet number 
$$Pe(h) := \frac{|\beta| h}{\mu},$$
the following error bound \cite[Theorem 1]{BNO20} was proven for the approximation $u_h$ in the diffusive regime $Pe(h)<1$,
\begin{equation}
\|u-u_h\|_{L^2(B)} \le C_{st}\, h^{\kappa} (\|u\|_{H^2(\Omega)} + h^{-1} \|\delta\|_{L^2(\omega)}),
\end{equation}
where the convergence order $\kappa \in (0,1)$ is the same as the H\"older exponent in \eqref{eq:continuum_stability} and the stability constant $C_{st}$ is proportional to the one in \eqref{eq:cstab}.
Under an additional assumption on the divergence of the convective field $\beta$, similar error bounds were also proven in the $H^1$-norm, see \cite[Theorem 2]{BNO20}. 

The prototypical effect of dominating diffusion is shown in \cref{fig:contour_diffusion}, where the problem is set in the unit square and contour error plots are shown for data assimilation from a centered disk of radius 0.1. One can notice oscillating errors that grow in size away from the data region towards the boundary. The exact solution in this example is $u=2\sin(5\pi x)\sin(5\pi y)$ where the factor 2 is taken to make its $L^2$-norm unitary. For the computation we used an unstructured mesh with 512 elements on a side and mesh size $h\approx0.0025$. 

\begin{figure}[h]
	\includegraphics[width=0.75\columnwidth]{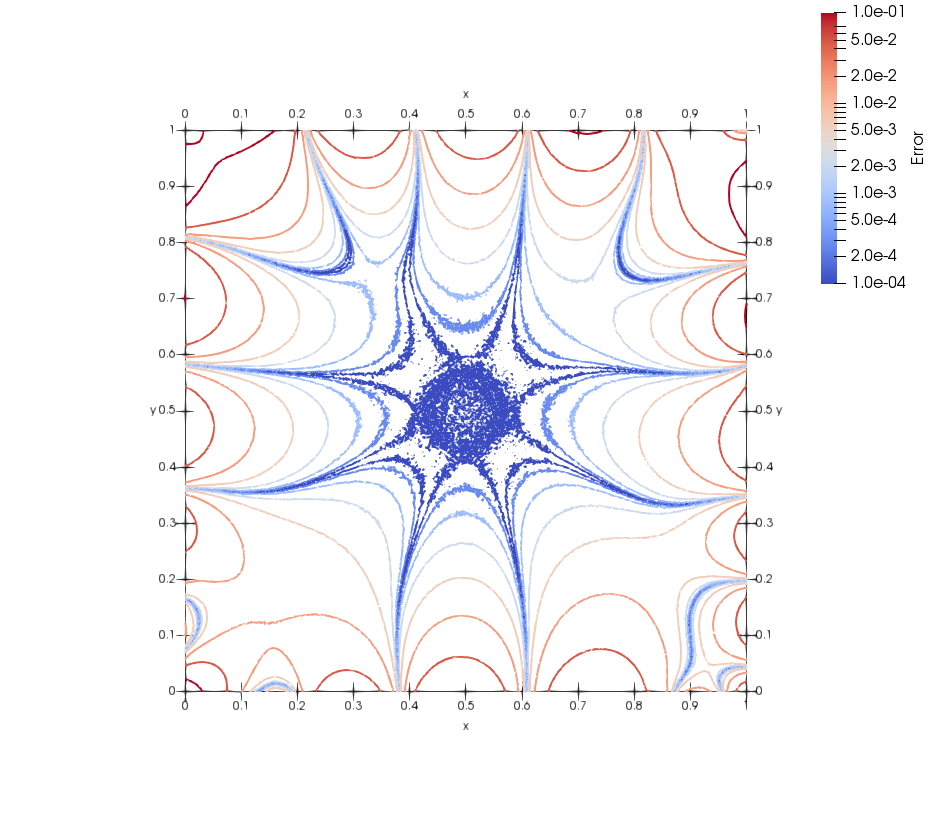}
	\caption{Absolute error contour plot in the diffusion-dominated case, $\mu = 1,\, \beta=(1,0)$. The domain is the unit square, data given in a centered disk of radius 0.1 for the exact solution $u=2\sin(5\pi x)\sin(5\pi y)$.}
	\label{fig:contour_diffusion}
\end{figure}

\subsection{Objective and main results}
%Added at review
We consider a stabilized finite element method for data assimilation subject to the convection--diffusion equation in the convection-dominated regime.
Since the behaviour of the physical system changes fundamentally when convection dominates and
$$Pe(h) \gg 1,$$
the goal of this second part paper is to reconsider the numerical method proposed in the first part \cite{BNO20} and develop an error analysis that captures and exploits the governing transport phenomenon.
This is illustrated in \cref{fig:contour_convection} where the transition to the convection-dominated regime through an intermediate regime is made by decreasing the diffusion coefficient $\mu$.
We aim to obtain sharper local error estimates along the characteristics of the convective field through the data region.
Even though the error analysis that we perform herein is different in nature to the one in the first part, the numerical method itself is only slightly changed (see \cref{rem:fem_penalty} below).
For the error localization technique we draw on ideas used for the streamline diffusion method in \cite{JNP84}, continuous interior penalty in \cite{BGL09}, and non-coercive hyperbolic problems in \cite{Bur14a}.

From the definition of the P\'eclet number we see that the regime will also depend on the resolution of the computation besides the physical parameters.
We can therefore expect the method to change behaviour as the resolution increases and $Pe(h)$ decreases.
This phenomenon was already observed computationally in \cite{Bur13} and can now be explained theoretically.

\begin{figure}[h]
	\begin{subfigure}{0.48\textwidth}
		\includegraphics[draft=false, width=\textwidth]{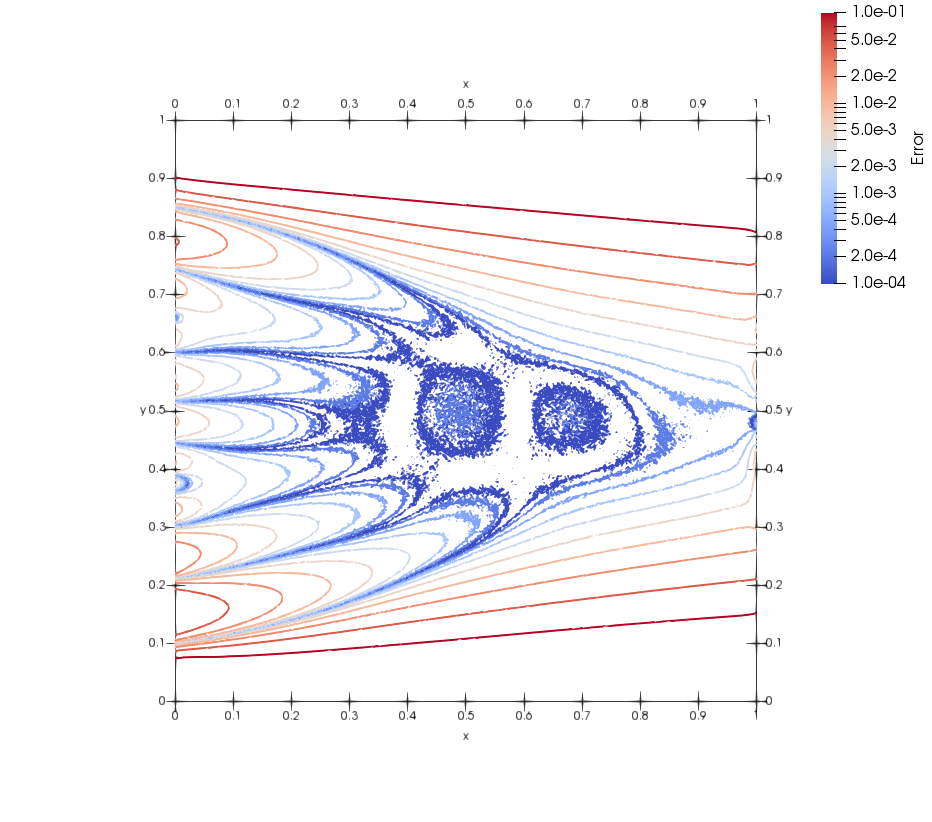}
		\caption{$\mu = 10^{-2}$.}
	\end{subfigure}
	\hfill
	\begin{subfigure}{0.48\textwidth}
		\includegraphics[draft=false, width=\textwidth]{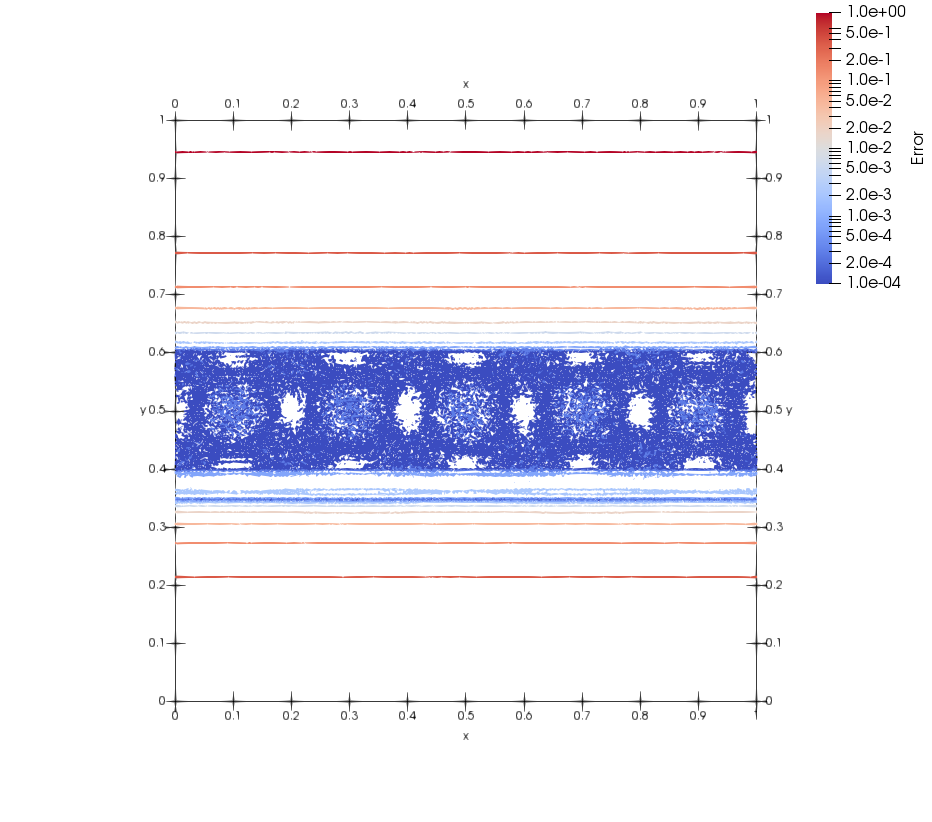}
		\caption{$\mu = 10^{-6}$.}
	\end{subfigure}
	\caption{Absolute error contour plot when convection becomes dominant, $\beta=(1,0)$. The domain is the unit square, data given in a centered disk of radius 0.1 for the exact solution $u=2\sin(5\pi x)\sin(5\pi y)$.}
	\label{fig:contour_convection}
\end{figure}

To make the presentation as simple as possible we consider a model case in the unit square $\Omega$ with constant convection
$$\beta := (\beta_1,0), \quad \beta_1\in \mathbb{R},$$
and the solution given in the subset
$$\omega := (0,x)\times (y^{-}, y^{+}) \text{ with } x > h \text{ and } y^{+} - y^{-} > h.$$
For the subset ${\omega_\beta} \subset \Omega$ covered by the characteristics of $\beta$ that go through $\omega$,
we introduce the stability region ${\mathring \omega_\beta} \subset \omega_\beta$ by cutting off a crosswind layer of width $\mathcal{O}(h^\frac12 |\ln(h)|)$ (see \cref{section_stab_weight} for more details).
%Added at review
We separate the convection-dominated and diffusion-dominated regimes by introducing a constant $Pe_{\lim} > 1$ such that
$$
Pe(h) > Pe_{\lim} > 1.
$$
To reduce the number of constants appearing in the analysis, we will write this as $Pe(h) \gtrsim 1$.
As suggested by \cref{fig:contour_convection}, we expect different results for data assimilation downstream vs upstream in an intermediate regime.
We prove in \cref{thm:error_weighted_downstream} weighted error estimates that for $\beta_1>0$ essentially take the following form
$$\|u - u_h\|_{L^2(\mathring \omega_\beta)} \leq C (|\beta|^{\frac12}
h^{\frac32} |u|_{H^2(\Omega)} + |\beta|^{\frac12} h^{-\frac12} \|\delta\|_{L^2(\omega)}),\text{ when } Pe(h) \gtrsim 1.$$
This is similar to the typical error estimates for piecewise linear stabilized FEMs for convection-dominated well-posed problems, such as local projection stabilization, dG methods, continuous interior penalty or Galerkin least squares.
On general shape-regular meshes these methods have an $\mathcal{O}(h^\frac12)$ gap to the best approximation convergence order. Taking this into account, our result is thus quasi-optimal.
For a recent overview of challenges and open problems in the well-posed case, see e.g. \cite{JKN2018} and \cite{RS15}.

When going against the characteristics, i.e. $\beta_1<0$, we prove in \cref{thm:error_weighted_upstream} first the pre-asymptotic bound
$$
\|u - u_h\|_{L^2(\mathring \omega_\beta)} \leq C ( |\beta|^{\frac12}
h |u|_{H^2(\Omega)} + h^{-1} \|\delta\|_{L^2(\omega)} ), \text{ when } 1 \lesssim Pe(h) < h^{-1},
$$
followed by
$$
\|u - u_h\|_{L^2(\mathring \omega_\beta)} \leq C (|\beta|^{\frac12}
h^{\frac32} |u|_{H^2(\Omega)} + h^{-\frac12} \|\delta\|_{L^2(\omega)}), \text{ when } Pe(h) > h^{-1}.
$$
It follows that when solving the data assimilation problem against the flow, the diffusivity reduces the convergence order in an
intermediate regime. Only for very small diffusion coefficients $\mu < |\beta| h^2$ do we get quasi-optimal bounds. This asymmetry of the error distribution for moderate P\'eclet numbers is clearly visible in the left plot of \cref{fig:contour_convection}.

Previous results on optimal control for stabilized convection--diffusion equations include \cite{BV07,DQ05,HYZ09,YZ09}. We refer to the first part \cite{BNO20} for a more detailed discussion.

In terms of notation, above and throughout the paper $C$ denotes a generic positive constant, not necessarily the same at each occurrence, that is independent of the coefficients $\mu,\, \beta$ and the mesh size $h$.

\section{Discrete setting}
Let $V_h\subset H^1(\Omega)$ be the conforming finite element space of piecewise affine $\mathbb{P}_1$ functions defined on a computational mesh $\mathcal{T}_h$ that consists
of shape-regular triangular elements $K$ with diameter $h_K$. The mesh size $h$ is the maximum over $h_K$ and we will assume that $h<1$.
The interior faces of all the elements are collected in the set $\mathcal{F}_i$ and the jump of a quantity across such a face $F$ is denoted by $\jump{\cdot}_F$, omitting the subscript whenever the context is clear. We denote by $n$ the unit normal.

First we introduce the standard inner products with the induced norms
$$
(v_h,w_h)_\Xi := \int_{\Xi} v_h w_h ~\mbox{d} x, \quad
\left<v_h,w_h\right>_{\partial \Xi} := \int_{\partial \Xi} v_h w_h ~\mbox{d} s,
$$
and the bilinear form in the weak formulation of \eqref{eq:intro_model_problem}
$$
a_h(v_h,w_h) := (\beta \cdot \nabla  v_h, w_h )_\Omega+
( \mu \nabla v_h , \nabla w_h  )_{\Omega}-
\left< \mu \nabla v_h \cdot n , w_h \right>_{\partial \Omega}.
$$
We will make use of the stabilizing bilinear forms
$$
 s_{\Omega}(v_h,w_h):= \gamma \sum_{F \in \mathcal{F}_i} \int_{F} h (\mu + |\beta|h) \jump{\nabla v_h \cdot n} \cdot \jump{\nabla w_h \cdot n}~\mbox{d}s,
$$
which will act on the discrete solution penalizing the jumps of its gradient across interior faces, and 
$$
s_*(v_h,w_h) :=\gamma_* \left( \left< (|\beta|+ \mu h^{-1})v_h, w_h \right>_{\partial \Omega}+ (\mu \nabla v_h, \nabla w_h )_\Omega+ s_\Omega(v_h,w_h) \right),
$$
where $\gamma$ and $\gamma_*$ are positive constants that can be heuristically chosen at implementation. They do not play a role in the convergence of the method and most of the time we will include them in the generic constant $C$.
For the data assimilation term we consider the scaled inner product in the data set $\omega$ given by
$$
s_{\omega}(v_h,w_h):=((|\beta| h^{-1} + \mu h^{-\zeta}) v_h, w_h)_{ \omega}, \quad \zeta \in [0,2].
$$
To this we add the stabilizing interior penalty term $s_\Omega$ to define for conciseness
$$
s(v_h,w_h) := s_\Omega(v_h,w_h) + s_{\omega}(v_h,w_h),
$$

The idea behind the computational method follows the discretize-then-optimize approach: we first discretize and then formulate the data assimilation problem as a PDE-constrained optimization problem with additional stabilizing terms. Apart from their stabilizing intake, these terms are also chosen such that they vanish at optimal rates. For an overview on this approach to ill-posed problems and more details on the desired properties of discrete stabilizers, we refer the reader to \cite{Bur13}. To be more precise, for an approximation $u_h \in V_h$ and a discrete Lagrange multiplier $z_h \in V_h$, we consider the functional
\begin{align*}
L_h(u_h,z_h) :={}& \tfrac12 s_\omega (u_h - \tilde U_\omega, u_h - \tilde U_\omega) + a_h(u_h,z_h) - (f,z_h)_\Omega \\
			&+ \tfrac12 s_\Omega(u_h,u_h) - \tfrac12 s_*(z_h,z_h),
\end{align*}
where the first term is measuring the misfit between $u_h$ and the known perturbed restriction $\tilde U_\omega = u\vert_{\omega} + \delta$, the next two terms are imposing the weak form of the PDE \eqref{eq:intro_model_problem} as a constraint, and the last two terms have stabilizing role and act only on the discrete level.

We look for the saddle points of the Lagrangian $L_h$ and analyse their convergence to the exact solution. Using the optimality conditions we obtain the finite element method for data assimilation subject to
\eqref{eq:intro_model_problem}, which reads as follows: find $(u_h,z_h)  \in [V_h]^2$ such that
\begin{equation}\label{eq:method_FEM}
\left\{\begin{array}{rcl}
a_h(u_h,w_h) - s_*(z_h,w_h) &=& (f,w_h)_\Omega\\
a_h(v_h,z_h) + s(u_h,v_h) &=&s_{\omega}(\tilde U_\omega,v_h)
\end{array} \right.
,\quad \forall (v_h,w_h)  \in [V_h]^2.
\end{equation}
Notice that the exact solution $u\in H^2(\Omega)$ (with noise $\delta \equiv 0$) and the dual variable $z \equiv 0$ satisfy \eqref{eq:method_FEM} since the gradient of the exact solution has no jumps across interior faces. Hence the Lagrange multiplier $z_h$ should converge to zero.

\begin{remark}\label{rem:fem_penalty}
	The same finite element method \eqref{eq:method_FEM} has been proposed in the first part \cite{BNO20} for the diffusion-dominated case; $s_\Omega$ and $s_*$ are exactly the stabilizing terms introduced there. However, herein we have increased the penalty coefficient in the data term $s_\omega$ from $|\beta| h + \mu$ to $|\beta| h^{-1} + \mu h^{-\zeta}$. We note that the bounds in \cite{BNO20} hold also for this stronger penalty term, but the sensitivity to perturbations in data increases by a factor of $h^{-1}$.
\end{remark}

\begin{proposition}\label{prop:unique}
	The finite element method \eqref{eq:method_FEM} has a unique solution  $(u_h,z_h) \in [V_h]^2$ and the Euclidean condition number $\mathcal{K}_2$ of the system matrix  satisfies $$\mathcal{K}_2 \le C h^{-4}.$$ 
	\begin{proof}
		The proof given in \cite[Proposition 2]{BNO20} holds verbatim.
	\end{proof}
\end{proposition}

\subsection{Stability region and weight functions}\label{section_stab_weight}

We will exploit the convective term of the PDE to obtain stability in the zone that connects through characteristics to the data region $\omega$.
The objective is to obtain weighted $L^2$-estimates in this region that are independent of $\mu$ (but not of the regularity of the exact solution) with the underlying assumption that $\mu \ll |\beta|$.
To this end we first define the subdomain where we can obtain stability (see \cref{fig:ex_domains} for a sketch) and some weight functions that will be used to define weighted norms.
These can be given in explicit form in the simple framework where $\beta = (\beta_1,0)$ and
$$\omega := (0,x)\times (y^{-}, y^{+}) \text{ with } x > h \text{ and } y^{+} - y^{-} > h.$$ 

Let $\omega_\beta$ denote the closed set of all the points $p \in \bar\Omega$ that can be reached through characteristics from $\omega$,
i.e. for which there exists $s \in \mathbb{R}$ such that $p+s\beta \in \partial\omega$. Similarly to the classical work \cite{JNP84}, we define the stability region ${\mathring \omega_\beta}$ by cutting off a crosswind layer from $\omega_\beta$, namely
\begin{equation}\label{eq:stability_region}
{\mathring \omega_\beta} := \{p \in {\omega_\beta}:\, \mbox{dist}(p,\Omega \setminus \omega_\beta) \ge c_\lambda h^\frac12 \ln(1/h)\},
\end{equation}
with the constant $c_\lambda$ to be made precise in the following. In our setting, we simply have that ${\mathring \omega_\beta} = [0,1]\times [\mathring y^{-}, \mathring y^{+}]$ for some $\mathring y^{+} > \mathring y^{-} > 0$.
The crosswind layer and its width are motivated by the subsequent construction of weight functions with a specific decay outside $\omega_\beta$.

\begin{figure}
	\resizebox{0.66\textwidth}{!}{
		\centering
		\begin{tikzpicture}
		\draw (0,0) rectangle (10,10);
		\draw (9.5,9.5) node[align=center] {\Huge $\Omega$};
		
		\filldraw[color=gray] (0,4) rectangle (2,7);
		\filldraw[pattern=north east lines] (0,4.5) rectangle (10,6.5);
		
		\draw (0,0) node[align=center, below] {\Large 0};
		\draw (2,0) node[align=center, below] {\Huge $x$};
		\draw (0,4) node[align=center, left] {\Huge $y^-$};
		\draw (0,7) node[align=center, left] {\Huge $y^+$};
		\draw (10,4.5) node[align=center, right] {\Huge $\mathring y^-$};
		\draw (10,6.5) node[align=center, right] {\Huge $\mathring y^+$};
		
		%\draw (1,5.5) node[align=center] {\Huge $\omega$};
		%\draw (6,5.5) node[align=center] {\Huge $\mathring \omega_\beta$};
		\draw (1,8) node[align=center, above] {\huge $(\beta_1,0)$};
		%\draw (0,2.5) node[align=center, left] {\Huge $\partial \Omega^-$};
		%\draw (10,2.5) node[align=center, right] {\Huge $\partial \Omega^+$};
		
		\draw[dashed] (0,4) -- (10,4);
		\draw[dashed] (0,7) -- (10,7);
		\draw[dashed] (2,0) -- (2,10);
		
		\draw[decorate, decoration={brace, raise=2pt}, thick] (2,7) -- (2,6.5);
		\draw (2.1,6.75) node[align=center, right] {$\mathcal{O}(h^\frac12 \log(1/h))$};
		
		\foreach \x in {0.5, 2.5, 4.5, 6.5, 8.5} {
			\draw[->,thick] (\x,1.3) -- +(1,0);
			\draw[->,thick] (\x,2.6) -- +(1,0);
			\draw[->,thick] (\x,9) -- ++(1,0);
			\draw[->,thick] (\x,8) -- ++(1,0);
		}
		\foreach \x in {2.5, 4.5, 6.5, 8.5} {
			\draw[->,thick] (\x,5.5) -- +(1,0);
		}
		\end{tikzpicture}}
	\caption{Data set $\omega$ (gray) and the stability region ${\mathring \omega_\beta}$ (hatched).}
	\label{fig:ex_domains}
\end{figure}
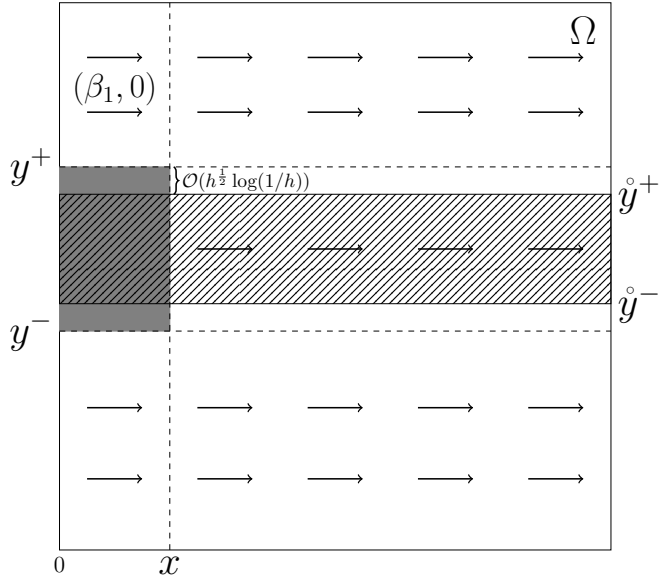

We will consider different weight functions depending on the direction of the convection field. In the downstream case we let
$$
\psi_1(x,y):= e^{-x}, \text{ when } \beta_1 >0,
$$
and in the upstream case
$$
\psi_1(x,y):= -e^{-x}, \text{ when } \beta_1 <0.
$$
In both cases we have that $\nabla \psi_1 = (-\psi_1, 0)$. Let then $\psi_2\in W^{1,\infty}(\Omega)$ be a function satisfying
\begin{equation}\label{eq:weight_psi2}
\psi_2=\left\{ \begin{array}{ll}
1, & \text{ in } \mathring \omega_\beta \\
\mathcal{O}(h^{3}), & \text{ in } \Omega \setminus \omega_\beta
\end{array}
\right.
,\quad \beta \cdot \nabla \psi_2=0, \quad |\nabla \psi_2| \leq C h^{-\frac12}.
\end{equation}
Such a function can be obtained by taking a positive constant $\lambda$ that will be specified later and letting
$$
\psi_2(x,y) := \left\{ 
\begin{array}{ll} 
\exp((\mathring y^{+} - y)/(\lambda h^\frac12)), & y > \mathring y^{+} \\
1, & (x,y) \in  {\mathring \omega_\beta} \\
\exp((y - \mathring y^{-})/(\lambda h^\frac12)), & y < \mathring y^{-}.
\end{array} 
\right.
$$
Note that $\psi_2$ is only piecewise continuously differentiable. For $\psi_2$ to decrease sufficiently rapidly to $\mathcal{O}(h^3)$ outside of $\omega_\beta$, we can take
$$
\mbox{dist}({\mathring \omega_\beta},\Omega \setminus \omega_\beta) = \min(y^{+} - \mathring y^{+}, \mathring y^{-} - y^{-})
\ge 3 \lambda h^\frac12 \ln(1/h),
$$
which corresponds to $c_\lambda = 3\lambda$ in the definition of $\mathring \omega_\beta$ given in \eqref{eq:stability_region}.
We thus have that
$$|\nabla \psi_2| \leq \lambda^{-1} h^{-\frac12},$$
and in the subsequent proofs the constant $\lambda$ will be taken large enough.

We now define the weight function $\varphi \in W^{1,\infty}(\Omega)$ that will be used in the weighted norms. For the downstream case we take in \cref{sec:error_estimates_downstream}
\begin{equation}\label{eq:weight_down}
\varphi:= \psi_1 \psi_2 \in (0,1), \text{ when } \beta_1 >0,
\end{equation}
and for the upstream case in \cref{sec:error_estimates_upstream},
\begin{equation}\label{eq:weight_up}
\varphi:= \psi_1 \psi_2 \in (-1,0), \text{ when } \beta_1 <0.
\end{equation}
Using the product rule and the fact that $\beta \cdot \nabla \psi_2=0$, it follows that in both cases we have
\begin{equation}\label{eq:beta_grad_weight}
\beta \cdot \nabla \varphi = -|\beta| |\varphi|,
\end{equation}
and
\begin{equation}\label{eq:grad_weight}
|\nabla \varphi| \leq (1 + \lambda^{-1} h^{-\frac12}) |\varphi|.
\end{equation}

We will denote the inflow and outflow boundaries by $\partial \Omega^-$ and $\partial \Omega^+$, i.e. $\beta \cdot n < 0$ on $\partial \Omega^-$ and $\beta \cdot n > 0$ on $\partial \Omega^+$. 
We will also assume that $\beta \cdot n=0$ can only hold on the boundary of $\Omega \setminus \omega_\beta$, and that $\mu \le \Pe^{-1}_{\lim} |\beta \cdot n|h$ when $\beta \cdot n \neq 0$.
This is straightforward to verify in the model case of the unit square that we are considering.

\section{Preliminaries and the discrete commutator property}
We first collect several inequalities that will be used repeatedly.
We recall the standard discrete inverse inequality
\begin{equation}\label{eq:inverse_ineq}
\| \nabla v_h\|_{L^2(K)} \le C h^{-1} \|v_h\|_{L^2(K)},\quad \forall v_h \in \mathbb{P}_1(K),
\end{equation}
see e.g. \cite[Lemma 1.138]{EG04}, the continuous trace inequality
\begin{equation}\label{eq:trace_cont}
\|v\|_{L^2(\partial K)} \le C( h^{-\frac12} \|v\|_{L^2(K)} + h^{\frac12} \|\nabla v\|_{L^2(K)}),\quad \forall v \in H^1(K),
\end{equation}
see e.g. \cite{MS99}, and the discrete trace inequality
\begin{equation}\label{eq:trace_grad}
\|\nabla v_h \cdot n\|_{L^2(\partial K)} \le C h^{-\frac12} \|\nabla v_h\|_{L^2(K)},\quad \forall v_h \in \mathbb{P}_1(K).
\end{equation}
We will use standard estimates for the $L^2$-projection $\pi_h:L^2(\Omega) \mapsto V_h$, namely
\begin{equation*}
	\norm{\pi_h u}_{H^m(\Omega)} \le C \norm{u}_{H^m(\Omega)},\quad u\in H^m(\Omega),\, m=0,1,
\end{equation*}
\begin{equation*}
\norm{u-\pi_h u}_{H^m(\Omega)} \le C h^{k-m} \norm{u}_{H^k(\Omega)},\quad u\in H^k(\Omega),\, k=1,2.
\end{equation*}
Scaling the result in \cite[Lemma 2]{BHL18} we recall the Poincar\'e-type inequality
\begin{equation}\label{eq:Poincare}
\|(\mu^{\frac12} h + |\beta|^{\frac12} h^{\frac32}) v_h\|_{H^1(\Omega)} \leq C \gamma^{-\frac12} s(v_h,v_h)^{\frac12}, \quad \forall v_h \in V_h.
\end{equation}
Using \eqref{eq:trace_cont} and approximation estimates we also have the jump inequality
\begin{equation}\label{eq:jumpineq}
s_\Omega(\pi_h u,\pi_h u)^{\frac12} \le C \gamma^{\frac12} (\mu^{\frac12} h + |\beta|^{\frac12} h^{\frac32})|u|_{H^2(\Omega)},\quad \forall u\in H^2(\Omega).
\end{equation}

We also recall that for a Lipschitz domain $K$ -- and hence for any element $K\in \mathcal{T}_h$ -- a function $\varphi$ is Lipschitz continuous if and only if $\varphi \in W^{1,\infty}(K)$. This follows from the proof in \cite[Theorem 4, p. 294]{Eva10} where the extension operator in the third step of the proof is replaced by the extension operator in \cite[Theorem 5, p. 181]{Stein70}. This equivalence holds for more general domains satisfying the minimal smoothness property in \cite[p. 189]{Stein70}. The proof in \cite[Theorem 4, p. 294]{Eva10}	also shows that the mean value theorem holds and for any $x,y\in K$,
\begin{align}\label{eq:mean_value}
|\varphi(x) - \varphi(y)| \le C_{ex} h_K |\varphi|_{W^{1,\infty}(K)},
\end{align}
where $|\varphi|_{W^{1,\infty}(K)} := \|\nabla \varphi\|_{\infty,K}$ and the constant $C_{ex}>0$ is the norm of the extension operator used.

\begin{lemma}\label{lem:weight}
	For all $v_h \in V_h$ and $K \in \mathcal{T}_h$, the following inequalities hold
	\begin{equation}\label{eq:lem_weight_1}
	\|\varphi\|_{\infty,K} \|v_h\|_K \leq C \| v_h \varphi \|_K,
	\end{equation}
	\begin{equation}\label{eq:lem_weight_trace}
	\| v_h \varphi \|_{\partial K} \leq C h^{-\frac12} \| v_h \varphi \|_K,
	\end{equation}
	assuming that $(h + \lambda^{-1} h^{\frac12})$ is small enough.
\end{lemma}
\begin{proof}
	Let $x^* \in K$ be such that
	$|\varphi(x^*)| = \|\varphi\|_{\infty,K}$. Using the triangle inequality we may write
	$$
	\|\varphi\|_{\infty,K} \|v_h\|_K \leq \|\varphi v_h\|_K + \|(\varphi(x^*) - \varphi) v_h\|_K.
	$$
	By the mean value theorem \eqref{eq:mean_value} we have that
	$$
	|\varphi(x^*) - \varphi| \leq C_{ex} h |\varphi|_{W^{1,\infty}(K)},
	$$
	and by \eqref{eq:grad_weight} together with the assumption that $C_{ex}(h + \lambda^{-1} h^{\frac12}) < \frac12$ we get
	$$
	C_{ex} h |\varphi|_{W^{1,\infty}(K)}
	\leq C_{ex} (h +\lambda^{-1} h^{\frac12}) \|\varphi\|_{\infty,K} \leq \frac12 \|\varphi\|_{\infty,K}.
	$$
	It follows that
	$$
	\|\varphi\|_{\infty,K} \|v_h\|_K \leq \|\varphi v_h\|_K + \frac12 \|\varphi\|_{\infty,K}\|v_h\|_K,
	$$
	from which the claim \eqref{eq:lem_weight_1} is immediate.
	Considering now \eqref{eq:lem_weight_trace}, using the standard element-wise trace inequality \eqref{eq:trace_cont} we have
	$$
	\|h^{\frac12} v_h  \varphi\|_{\partial K} \leq C (\|v_h  \varphi\|_{ K} + h \|\nabla (v_h  \varphi)\|_K).
	$$
	We bound the gradient term using \eqref{eq:grad_weight} and the inverse inequality \eqref{eq:inverse_ineq},
	\begin{align*}
	h\|\nabla (v_h \varphi)\|_K &\leq h\|v_h \nabla \varphi\|_K + h\|\varphi \nabla v_h\|_K \\
	& \leq (h + \lambda^{-1} h^{\frac12}) \|\varphi\|_{\infty,K} \|v_h \|_K + C \|\varphi\|_{\infty,K} \|v_h\|_K.
	\end{align*}
	We conclude by collecting the terms and using \eqref{eq:lem_weight_1}.
\end{proof}

\subsection{Discrete commutator property}
We denote by $i_h$ the Lagrange nodal interpolant. We herein consider a Lipschitz weight function and prove the following super approximation result, also known as the discrete commutator property. This result will be essential to derive local weighted estimates and is similar to the one proven in \cite{Ber99} for smooth compactly supported weight functions.
For an introduction to interior estimates we refer the reader to \cite{NS74}.
\begin{lemma}\label{lem:disc_com}
Let $v_h \in V_h$ and $K\in \mathcal{T}_h$. Then for any weight function $\varphi \in W^{1,\infty}(K)$
\begin{equation*}
\|v_h \varphi - i_h (v_h \varphi)\|_K + h \|\nabla(v_h \varphi - i_h (v_h \varphi))\|_K
\leq C h |\varphi|_{W^{1,\infty}(K)} \|v_h\|_K.
\end{equation*}	
\begin{proof}
We will first show the $L^2$-norm estimate
$$
\|v_h \varphi - i_h (v_h \varphi)\|_{K} \leq C h  |\varphi|_{W^{1,\infty}(K)} \|v_h\|_{K}. 
$$
Let $x^* \in K$ be such that $|\varphi(x^*)| = \|\varphi\|_{\infty,K}$ and let $R_\varphi = \varphi - \varphi(x^*)$.
Note that
$$
\|(1-i_h) (v_h \varphi)\|_{K}  = \|(1-i_h) (v_h R_\varphi)\|_{K}.
$$ 
Observe that $i_h (v_h \varphi) = i_h (v_h i_h\varphi)$ and therefore
$$
\|(1-i_h) (v_h R_\varphi)\|_{K}  = \|v_h  R_\varphi - i_h (v_h  i_h R_\varphi)\|_{K}.
$$
By the triangle inequality
$$
\|i_h (v_h i_h R_\varphi) - v_h  R_\varphi\|_{K} \leq \|i_h (v_h i_h R_\varphi) - v_h i_h R_\varphi\|_{K} + \|v_h ( i_h R_\varphi-  R_\varphi)\|_{K}.
$$
For the first term, since $v_h i_h R_\varphi \in H^1(K)$ we have by standard interpolation that
$$
\|i_h (v_h i_h R_\varphi) - v_h i_h R_\varphi\|_{K} \leq C h \|\nabla (v_h i_h R_\varphi)\|_{K}
$$
and then
$$
\|\nabla (v_h i_h R_\varphi)\|_{K} \leq |i_h R_\varphi|_{W^{1,\infty}(K)} \|v_h\|_{K} + \|i_h R_\varphi\|_{\infty,K} \|\nabla v_h\|_{K}.
$$
By inserting $\nabla \varphi$ and $\varphi$, respectively, and using interpolation estimates in $W^{1,\infty}(K)$ \cite[Theorem 1.103]{EG04} and the mean value theorem \eqref{eq:mean_value} we have the following approximation
$$
h |i_h R_\varphi|_{W^{1,\infty}(K)} + \|i_h R_\varphi\|_{\infty,K} \leq C h |\varphi|_{W^{1,\infty}(K)}.
$$
Combined with the previous estimate and the inverse inequality \eqref{eq:inverse_ineq} this gives that
\begin{equation}\label{eq:grad_R}
\|\nabla (v_h i_h R_\varphi)\|_{K} \leq  C|\varphi|_{W^{1,\infty}(K)} \|v_h\|_{K}.
\end{equation}
For the second term, using again interpolation \cite[Theorem 1.103]{EG04} we have
$$
\|v_h ( i_h R_\varphi-  R_\varphi)\|_{K} \leq \|i_h R_\varphi - R_\varphi\|_{\infty,K} \|v_h\|_{K}
\leq C h |\varphi|_{W^{1,\infty}(K)} \|v_h\|_{K},
$$
and thus we have shown that
$$
\|v_h \varphi - i_h (v_h \varphi)\|_{K} \leq C h  |\varphi|_{W^{1,\infty}(K)} \|v_h\|_{K}.
$$
The approximation estimate for the gradient follows by the same arguments. Indeed,
\begin{align*}
\|\nabla(1 - i_h) (v_h \varphi)\|_K &= \|\nabla(1 - i_h) (v_h R_\varphi)\|_K = \|\nabla (v_h R_\varphi) - \nabla i_h(v_h i_h R_\varphi)\|_K \\
&\leq \|\nabla (v_h R_\varphi) - \nabla (v_h i_h R_\varphi)\|_K
+ \|\nabla (v_h i_h R_\varphi) - \nabla i_h(v_h i_h R_\varphi)\|_K.
\end{align*}
We first use interpolation and the inverse inequality \eqref{eq:inverse_ineq} to get
\begin{align*}
\|\nabla (v_h ( R_\varphi - i_h R_\varphi))\|_K
&\leq \|v_h \nabla (R_\varphi - i_h R_\varphi)\|_K + \| (R_\varphi - i_h R_\varphi) \nabla v_h \|_K \\
&\leq C |\varphi|_{W^{1,\infty}(K)} \|v_h\|_{K}.
\end{align*}
Then we use an inverse inequality followed by interpolation and \eqref{eq:grad_R} to obtain
$$
\|\nabla (v_h i_h R_\varphi) - \nabla i_h(v_h i_h R_\varphi)\|_K
\leq C \|\nabla (v_h i_h R_\varphi)\|_{K}
\leq  C|\varphi|_{W^{1,\infty}(K)} \|v_h\|_{K}.
$$
	
\end{proof}
\end{lemma}

\section{A priori local error estimates}\label{sec:error_estimates}
\subsection{Consistency and continuity}
The following consistency result holds exactly as in the diffusion-dominated case, see \cite[Lemma 4]{BNO20}. We give the proof for the sake of completeness.
\begin{lemma}[Consistency]\label{lem:consist}
	Let $u \in H^2(\Omega)$ be the solution to
	\eqref{eq:intro_model_problem} and $(u_h,z_h) \in [V_h]^2$ the solution to \eqref{eq:method_FEM}, then
	$$
	a_h(\pi_h u - u_h,w_h) + s_*(z_h,w_h) = a_h(\pi_h u - u,w_h),
	$$
	and
	$$
	-a_h(v_h,z_h) + s(\pi_h u - u_h,v_h) = s_\Omega(\pi_h u - u,v_h) +
	s_\omega(\pi_h u - \tilde U_\omega,v_h),
	$$
	for all $(v_h,w_h)\in [V_h]^2$.
	\begin{proof}
		The first claim follows from the definition of $a_h$, since
		\begin{equation*}
		a_h(u_h,w_h) - s_*(z_h,w_h) = (f,w_h)_\Omega = (-\mu \Delta u + \beta \cdot \nabla u, w_h)_\Omega = a_h(u,w_h),
		\end{equation*}
		where in the last equality we integrated by parts.
		The second claim follows similarly from
		$$
		a_h(v_h,z_h) + s(u_h,v_h) = 
		s_\omega(\tilde U_\omega,v_h), 
		$$
		which combined with the fact that $s_\Omega(u,v_h)=0$ leads to
		\begin{align*}
		-a_h(v_h,z_h) + s(\pi_h u - u_h,v_h) &= s(\pi_h u,v_h) - s_\omega(\tilde
		U_\omega,v_h)  \\
		&=  s_\Omega(\pi_h u - u,v_h) + s_\omega(\pi_h u - \tilde U_\omega,v_h).
		\end{align*}
	\end{proof}
\end{lemma}
We now introduce the stabilization norm on $[V_h]^2$ by combining the primal and dual stabilizers
$$
\|(v_h,w_h)\|^2_s:= s(v_h,v_h) + s_*(w_h,w_h),
$$
and the ``continuity norm'' defined on $H^{\frac32+\varepsilon}(\Omega)$, for any $\varepsilon>0$,
$$
\| v \|_\sharp:= \||\beta|^\frac12 h^{-\frac12} v \|_{\Omega}
+ \||\beta|^{\frac12} h^{\frac12} \nabla v\|_\Omega + \|h^{\frac12} \mu^{\frac12} \nabla v \cdot n\|_{\partial \Omega}.
$$
From the jump inequality \eqref{eq:jumpineq}, standard approximation bounds for $\pi_h$ and the trace inequality \eqref{eq:trace_cont}, it follows that for $u \in H^2(\Omega)$
\begin{equation}\label{eq:approx}
\|(u - \pi_h u,0)\|_s + \| u - \pi_h u \|_\sharp
\le C (\mu^{\frac12} h +|\beta|^{\frac12} h^{\frac32}) |u|_{H^2(\Omega)}.
\end{equation}

We also define the orthogonal space 
$$
V_h^\perp := \{ v \in H^2(\Omega): (v,w_h)_\Omega = 0,\quad
\forall w_h \in V_h\}.
$$

\begin{lemma}\label{lem:cont}(Continuity)
Let $v \in V_h^{\perp}$ and $w_h \in V_h$, then
$$
a_h(v,w_h) \leq C \| v \|_\sharp \|(0,w_h)\|_s.
$$
\end{lemma}
\begin{proof}
Integrating by parts in the convective term of $a_h$ and using $\nabla \cdot \beta = 0$ leads to
$$
a_h(v,w_h) = -(v,\beta \cdot \nabla w_h)_\Omega + \left<v \beta \cdot n,w_h \right>_{\partial \Omega}
+ ( \mu \nabla v , \nabla w_h  )_{\Omega} -\left< \mu \nabla v \cdot n , w_h \right>_{\partial \Omega}.
$$
For the first term we recall the discrete approximation estimate that holds for any piecewise linear $\beta$, see e.g. \cite[Theorem 2.2]{Bur05},
\begin{equation}\label{eq:beta_approx}
	\begin{split}
		\inf_{x_h \in V_h} \|h^{\frac12} (\beta \cdot \nabla w_h - x_h)\|_\Omega
		&\leq C \left(\sum_{F \in \mathcal{F}_i} \|h \jump{\beta \cdot \nabla w_h}\|_{F}^2\right)^{\frac12} \\
		&\leq C |\beta|^\frac12 \gamma^{-\frac12} s_\Omega(w_h,w_h)^{\frac12}		
	\end{split}
\end{equation}
and use orthogonality to obtain
\begin{equation*}
-(v,\beta \cdot \nabla w_h)_\Omega \leq \|h^{-\frac12} v\|_\Omega \inf_{x_h \in V_h} \|h^{\frac12} (\beta \cdot \nabla w_h -x_h)\|_\Omega \leq C \| v \|_\sharp \|(0,w_h)\|_s.
\end{equation*}
For the remaining terms, applying the Cauchy-Schwarz inequality we see that
$$
\left<v \beta \cdot n,w_h \right>_{\partial \Omega}+ ( \mu \nabla v , \nabla w_h  )_{\Omega}-
\left< \mu \nabla v \cdot n , w_h \right>_{\partial \Omega} \leq C \| v \|_\sharp \|(0,w_h)\|_s.
$$
\end{proof}

Note that the proof of the above continuity estimate holds for any divergence-free piecewise linear velocity field $\beta$.
To address the case of a general velocity field $\beta\in W^{1,\infty}(\Omega)$ one can use a similar argument by considering its piecewise linear approximation as in \cite[Lemma 5]{BNO20}. Assuming that $\beta$ is divergence-free, the constant would be proportional to $h Pe(h)^\frac12 |\beta|_{1,\infty} / \|\beta\|_{\infty}$, otherwise it would be proportional to $Pe(h)^\frac12 |\beta|_{1,\infty} / \|\beta\|_{\infty}$.
\subsection{Convergence of regularization}
We now prove optimal convergence for the stabilizing and data assimilation terms.
\begin{proposition}(Convergence of regularization).\label{prop:conv_stab}
Let $u \in H^2(\Omega)$ be the solution of \eqref{eq:intro_model_problem}
and $(u_h,z_h) \in [V_h]^2$ the solution to \eqref{eq:method_FEM}, then there holds
$$
\|(\pi_h u - u_h,z_h)\|_s \leq C(\mu^{\frac12} h + |\beta|^{\frac12} h^{\frac32}) (|u|_{H^2(\Omega)} + h^{-2}\|\delta\|_\omega).
$$
\end{proposition}
\begin{proof}
Denoting $e_h = \pi_h u - u_h$ we have that
$$
\|(e_h,z_h)\|_s^2 = a_h(e_h,z_h)  + s_*(z_h,z_h) - a_h(e_h,z_h) + s(e_h,e_h).
$$
Using both claims in \cref{lem:consist} we may write
$$
\|(e_h,z_h)\|_s^2 = a_h(\pi_h u - u,z_h) +s_\Omega(\pi_h u -
u,e_h)+s_\omega(\pi_h u - \tilde U_\omega,e_h).
$$
Since $\pi_h u - u \in V_h^\perp$ we have by \cref{lem:cont} that
$$
 a_h(\pi_h u - u,z_h) \leq C \|\pi_h u - u\|_\sharp \|(0,z_h)\|_s.
$$
We bound the other terms using the Cauchy-Schwarz inequality
\begin{equation*}
s_\Omega(\pi_h u - u,e_h)+s_\omega(\pi_h u - \tilde U_\omega,e_h)
\leq (\|(\pi_h u - u,0)\|_s + (|\beta| h^{-1} + \mu h^{-\zeta})^{\frac12} \|\delta\|_\omega) \|(e_h,0)\|_s.
\end{equation*}
Collecting the above bounds we have
\begin{equation*}
\|(e_h,z_h)\|_s^2 \\
\leq C(\|\pi_h u - u\|_\sharp+\|(\pi_h u -
u,0)\|_s + (|\beta| h^{-1} +
\mu h^{-\zeta})^{\frac12} \|\delta\|_\omega) \|(e_h,z_h)\|_s
\end{equation*}
and the claim follows by applying the approximation inequality \eqref{eq:approx}.
\end{proof}

\begin{remark}
	Compared to the result in the diffusion-dominated case \cite[Proposition 3]{BNO20}, the sensitivity to data perturbations has increased by a factor of $h^{-1}$. This is due to the stronger penalty in the data term $s_\omega$ (c.f. \cref{rem:fem_penalty}). 
\end{remark}

\subsection{Downstream estimates}\label{sec:error_estimates_downstream}
In this case we consider $\beta = (\beta_1,0)$ with $\beta_1 > 0$ and the data set $$\omega = (0,x)\times (y^{-}, y^{+})$$ touching part of the inflow boundary $\partial \Omega^-$. We aim to obtain control of the following weighted triple norm defined on $V_h$,
\begin{equation}\label{eq:triple_norm_down}
\tnorm{v_h}_\varphi^2 := \| |\beta|^\frac12 v_h \varphi^{\frac12}\|_{\Omega}^2 +
\| \mu^{\frac12} \nabla v_h \varphi^{\frac12}\|_{\Omega}^2 + \||\beta \cdot n|^{\frac12} v_h
\varphi^{\frac12}\|_{ \partial \Omega^+}^2,
\end{equation}
where $\varphi$ is given by \eqref{eq:weight_down}. Since $\varphi \in (0,1)$, we will often use that $\| \cdot \varphi \|_\Omega \leq \| \cdot \varphi^\frac12 \|_\Omega$. We first consider $v_h \varphi$ as a test function in the weak bilinear form $a_h$ and obtain the following bound.
\begin{lemma}\label{lem:weight_control}
	There exist $\alpha>0$ and $h_0>0$ such that for all $h<h_0$ and $v_h \in V_h$ we have
	$$
	\alpha \tnorm{v_h}_\varphi^2 \leq  a_h(v_h, v_h \varphi) + C \|(v_h,0)\|^2_s.
	$$
\begin{proof}
We start with the convective term. Since $\nabla \cdot \beta = 0$, the divergence theorem leads to
\begin{equation*}
2 (\beta \cdot \nabla v_h, v_h \varphi)_\Omega = \left< v_h \beta \cdot n, v_h \varphi \right>_{\partial \Omega}
- (v_h, v_h \beta \cdot \nabla \varphi)_\Omega.
\end{equation*}
Then combining with \eqref{eq:beta_grad_weight} we have
\begin{equation*}
(\beta \cdot \nabla v_h, v_h \varphi)_\Omega = \frac12 \left( \left< v_h \beta \cdot n, v_h \varphi \right>_{\partial \Omega}   + |\beta| \|v_h \varphi^{\frac12}\|^2_\Omega \right).
\end{equation*}
We split the boundary term into inflow and outflow
$$
\left< v_h \beta \cdot n, v_h \varphi \right>_{\partial \Omega} = -\| |\beta \cdot n|^{\frac12} v_h \varphi^{\frac12} \|_{\partial \Omega^-}^2
+ \| |\beta \cdot n|^{\frac12} v_h \varphi^{\frac12} \|_{\partial \Omega^+}^2,
$$
and write
$$
\frac12 \left( \| |\beta \cdot n|^{\frac12} v_h \varphi^{\frac12} \|_{\partial \Omega^+}^2  + |\beta| \|v_h \varphi^{\frac12}\|^2_\Omega \right)
=(\beta \cdot \nabla v_h, v_h \varphi)_\Omega + \frac12 \| |\beta \cdot n|^{\frac12} v_h \varphi^{\frac12} \|_{\partial \Omega^-}^2.
$$
Splitting now the inflow boundary with respect to the closed set $\omega_\beta$ and using the discrete trace inequality \eqref{eq:trace_cont} in  $\omega$, and the weight decay \eqref{eq:weight_psi2} together with a standard global trace inequality for $H^1$ functions outside, we have that
\begin{align}
\begin{split}\label{eq:bound_inflow}
\| |\beta \cdot n|^{\frac12} v_h \varphi^{\frac12}\|_{\partial \Omega^-}
&\leq C |\beta|^{\frac12} (\| v_h \varphi^{\frac12}\|_{\partial \Omega^- \cap \omega_\beta} + \| v_h \varphi^{\frac12}\|_{\partial \Omega^- \setminus \omega_\beta})\\
&\leq C |\beta|^{\frac12} h^{-\frac12} \| v_h \|_{\omega} +  C |\beta|^{\frac12} h^{\frac32} \|v_h\|_{H^1(\Omega)} \\
&\leq C \gamma^{-\frac12} \|(v_h,0)\|_s,
\end{split}
\end{align}
where in the last step we used the Poincar\'e-type inequality \eqref{eq:Poincare}.
Hence we obtain control over the convective terms in the triple weighted norm
\begin{equation}\label{eq:triple_conv_stab}
\frac12 \left( \| |\beta \cdot n|^{\frac12} v_h \varphi^{\frac12} \|_{\partial \Omega^+}^2  + |\beta| \|v_h \varphi^{\frac12}\|^2_\Omega \right)
\leq  (\beta \cdot \nabla v_h, v_h \varphi)_\Omega + C \gamma^{-1} \|(v_h,0)\|^2_s.
\end{equation}
Let us consider the terms in $a_h(v_h, v_h \varphi)$ corresponding to the diffusion operator, starting with
$$
(\mu \nabla v_h,\nabla (v_h \varphi))_{\Omega}
= \|\mu^{\frac12} \nabla v_h \varphi^{\frac12}\|_{\Omega}^2 + (\mu \nabla v_h, v_h \nabla \varphi)_\Omega,
$$
which we rearrange as
$$
\|\mu^{\frac12} \nabla v_h \varphi^{\frac12}\|_{\Omega}^2
= (\mu \nabla v_h,\nabla (v_h \varphi))_{\Omega}
- (\mu \nabla v_h, v_h \nabla \varphi)_\Omega. 
$$
Bounding $\nabla \varphi$ by \eqref{eq:grad_weight} and using Cauchy-Schwarz together with $\mu \le |\beta| h$ we have that
\begin{align*}
| (\mu \nabla v_h, v_h \nabla \varphi)_\Omega |
&\leq \mu(|\nabla v_h \cdot \nabla \varphi|, v_h)_\Omega \\ 
&\leq \mu (1 + \lambda^{-1} h^{-\frac12}) (|\nabla v_h| \varphi^{\frac12}, v_h \varphi^{\frac12})_{\Omega}\\
&\leq \frac13 \|\mu^{\frac12} \nabla v_h \varphi^{\frac12}\|_{\Omega}^2
+ C(h + \lambda^{-2}) |\beta| \|v_h \varphi^{\frac12}\|^2_{\Omega}.
\end{align*}
We split the boundary term $\left< \mu \nabla v_h \cdot n , v_h \varphi\right>_{\partial \Omega}$ into inflow and outflow again. Similarly to \eqref{eq:bound_inflow} we have that
$$
\left< \mu \nabla v_h \cdot n, v_h \varphi\right>_{ \partial \Omega^-}  \leq C h \gamma^{-1} \|(v_h,0)\|^2_s. 
$$
For the outflow boundary term we use Cauchy-Schwarz and a trace inequality to obtain
\begin{align*}
	\left< \mu \nabla v_h \cdot n, v_h \varphi\right>_{ \partial \Omega^+}
	&\leq \|\mu^{\frac12} \nabla v_h \cdot n \varphi^{\frac12}\|_{ \partial \Omega^+} \|\mu^{\frac12} v_h \varphi^{\frac12}\|_{ \partial \Omega^+} \\
	&\leq C h^{-\frac12} \|\mu^{\frac12} \nabla v_h \varphi^{\frac12}\|_{\Omega} Pe_{\lim}^{-\frac12} h^{\frac12} \||\beta \cdot n|^{\frac12} v_h \varphi^{\frac12}\|_{ \partial \Omega^+} \\
	&\leq {\textstyle\frac13} \big\|\mu^{\frac12} \nabla v_h \varphi^{\frac12}\big\|_{\Omega}^2 + Pe^{-1}_{\lim} \big\||\beta \cdot n|^{\frac12} v_h \varphi^{\frac12}\big\|_{ \partial \Omega^+}^2.
\end{align*}
We denote the part of the boundary where $\beta \cdot n = 0$  by $\partial \Omega^0$ and use the assumption that $\partial \Omega^0$ is away from $\omega_\beta$, meaning that the weight function $\varphi$ is $\mathcal{O}(h^3)$ there.
We use trace inequalities and \eqref{eq:Poincare} to bound
$$
\left< \mu \nabla v_h \cdot n , v_h \varphi\right>_{\partial \Omega^0} \leq C \gamma^{-1} \|(v_h,0)\|^2_s.
$$
Collecting the above bounds we obtain that
\begin{align*}\label{eq:triple_diff_stab}
\tfrac13 \|\mu^{\frac12} \nabla v_h \varphi^{\frac12}\|_{\Omega}^2
\leq{}& (\mu \nabla v_h,\nabla (v_h \varphi))_{\Omega}  - \left< \mu \nabla v_h \cdot n , v_h \varphi\right>_{\partial \Omega}\\
&+ C (h + \lambda^{-2} + Pe^{-1}_{\lim}) ( \| |\beta|^{\frac12} v_h \varphi^{\frac12}\|_{\Omega}^2 + \||\beta \cdot n|^{\frac12} v_h
\varphi^{\frac12}\|^2_{ \partial \Omega^+})\\
&+ C \gamma^{-1} \|(v_h,0)\|^2_s.
\end{align*}
We conclude by combining this with \eqref{eq:triple_conv_stab} and assuming that $h$ is small enough and $Pe_{\lim}$ are large enough (thus absorbing the convective terms from the rhs into the lhs).
\end{proof}
\end{lemma}

Now we refine the control over the triple norm $\tnorm{v_h}_\varphi$ by taking the projection $\pi_h (v_h \varphi) \in V_h$ as a test function. 

\begin{corollary}\label{cor:weight_control_projection}
	There exist $\alpha>0$ and $h_0>0$ such that for all $h<h_0$ and $v_h \in V_h$ we have
	\begin{equation*}
	\alpha \tnorm{v_h}_\varphi^2  \leq a_h(v_h, \pi_h (v_h\varphi) ) + C \|(v_h,0)\|^2_s.
	\end{equation*}
\begin{proof}
Since
$$
a_h(v_h, \pi_h (v_h \varphi) ) = a_h(v_h, (\pi_h-1) (v_h \varphi) ) + a_h(v_h, v_h \varphi),
$$
we must bound $a_h(v_h, (\pi_h-1) (v_h \varphi) )$ in a suitable way.
Writing out the terms we have
\begin{align*}
a_h(v_h, (\pi_h-1) (v_h \varphi) ) ={}& (\beta \cdot \nabla v_h, (\pi_h-1) (v_h
\varphi) )_{\Omega} + (\mu \nabla v_h, \nabla (\pi_h-1) (v_h
\varphi) )_{\Omega} \\
&- \left<\mu \nabla v_h \cdot n,  (\pi_h-1) (v_h
\varphi) \right>_{\partial \Omega} = I + II + III.
\end{align*}
Let us consider the convection term first, and use orthogonality combined with \eqref{eq:beta_approx}
\begin{align*}
I = (\beta \cdot \nabla v_h, (\pi_h-1) (v_h\varphi) )_{\Omega}
&\leq C |\beta|^\frac12 \gamma^{-\frac12} \|(v_h,0)\|_s  h^{-\frac12} \|(\pi_h-1) (v_h\varphi)\|_{\Omega} \\
&\leq C |\beta|^\frac12 \gamma^{-\frac12} \|(v_h,0)\|_s h^{-\frac12} \|(i_h-1) (v_h\varphi)\|_{\Omega}.
\end{align*}
Integrating by parts and using that $\Delta v_h = 0$ on every element $K$ we obtain by the trace inequality \eqref{eq:trace_cont} and the assumption $Pe(h)>1$ that
\begin{align*}II + III &= \sum_{F \in \mathcal{F}_i} \int_{F} \mu \jump{\nabla v_h \cdot n} (\pi_h-1) (v_h
\varphi) ~\mbox{d}s \\
&\leq C |\beta|^{\frac12} \gamma^{-\frac12} s_\Omega(v_h,v_h)^{\frac12} (h^{-\frac12} \| (\pi_h-1)(v_h\varphi)\|_{\Omega}
+ h^{\frac12} \|\nabla  (\pi_h-1) (v_h\varphi)\|_{\Omega}).
\end{align*}
Notice that $i_h (v_h\varphi) = \pi_h(i_h (v_h\varphi))$ and the stability of the projection gives
\begin{equation}\label{eq:gradient_projection_interpolator}
\|\nabla (\pi_h-i_h) (v_h\varphi)\|_{\Omega} = \|\nabla \pi_h(1-i_h) (v_h\varphi)\|_{\Omega}
\leq C \|\nabla (1-i_h) (v_h\varphi)\|_{\Omega},
\end{equation}
and hence
\begin{equation}
\begin{split}\label{eq_gradient_projection}
h^{\frac12} \|\nabla  (\pi_h-1) (v_h\varphi)\|_{\Omega}
&\leq h^{\frac12} (\|\nabla (\pi_h-i_h) (v_h\varphi)\|_{\Omega}+\|\nabla  (i_h-1) (v_h\varphi)\|_{\Omega}) \\ 
&\leq C h^{\frac12} \|\nabla  (i_h-1) (v_h \varphi)\|_{\Omega}.
\end{split}
\end{equation}
Since 
$$
h^{-\frac12} \| (\pi_h-1) (v_h
\varphi)\|_{\Omega} \leq C h^{-\frac12} \|(i_h-1) (v_h\varphi)\|_{\Omega},
$$
collecting the contributions above we see that
$$
I+II+III \leq C |\beta|^{\frac12} \gamma^{-\frac12} \|(v_h,0)\|_s
\underbrace{(h^{-\frac12} \|(i_h-1) (v_h\varphi)\|_{\Omega} + h^{\frac12}  \|\nabla  (i_h-1) (v_h\varphi)\|_{\Omega})}_{IV},
$$
and hence
$$
a_h(v_h, (\pi_h-1) (v_h \varphi)) = I+II+III \le C \gamma^{-1} \|(v_h,0)\|^2_s + |\beta| IV^2.
$$
The discrete commutator property \cref{lem:disc_com} together with the $\varphi$-bounds \eqref{eq:grad_weight} and \eqref{eq:lem_weight_1} give that
\begin{equation}\label{eq:com_weight}
IV \leq C h^{\frac12} \|\nabla \varphi\|_{\infty,\Omega} \|v_h\|_\Omega
\leq C (h^{\frac12}+\lambda^{-1}) \| v_h \varphi \|_\Omega.
\end{equation}
Since $\varphi \in (0,1)$ and $\varphi < \varphi ^\frac12$, it follows that for $h$ small enough and $\lambda$ large enough, given some $\alpha>0$ we have
\begin{equation}\label{eq:com_weight_beta}
|\beta| IV^2 \leq \frac{\alpha}{2} \tnorm{v_h}_\varphi^2.
\end{equation}
Collecting the estimates for $a_h(v_h, (\pi_h-1) (v_h \varphi) )$ and using \cref{lem:weight_control} we see that
\begin{align*}
a_h(v_h, \pi_h (v_h \varphi) ) = a_h(v_h, (\pi_h-1) (v_h \varphi) ) + a_h(v_h, v_h \varphi)
\ge \frac{\alpha}{2} \tnorm{v_h}_\varphi^2 - C \gamma^{-1} \|(v_h,0)\|^2_s,
\end{align*}
from which we conclude by renaming $\alpha/2$ as $\alpha$.  
\end{proof}
\end{corollary}

\begin{lemma}\label{lem:weight_control_lb}
For all $v_h \in V_h$ there holds
\begin{equation*}
\|(0,\pi_h (v_h \varphi))\|^2_s \leq C (\tnorm{v_h}^2_\varphi + \|(v_h,0)\|^2_s).
\end{equation*}
\end{lemma}
\begin{proof}
First note that by triangle inequalities we have that up to a constant
\begin{align*}
 \|(0,\pi_h (v_h \varphi))\|_s \leq{}& \|\mu^{\frac12} \nabla (\pi_h - 1) (v_h
\varphi))\|_\Omega + \|\mu^{\frac12} \nabla (v_h
\varphi))\|_\Omega \\
&+ (|\beta|+ \mu h^{-1})^{\frac12} (\|(\pi_h-1) (v_h \varphi)\|_{\partial \Omega}
+ \|v_h\varphi\|_{\partial \Omega} )\\
&+ s_\Omega(\pi_h (v_h \varphi), \pi_h (v_h \varphi))^{\frac12}.
\end{align*}
We bound these terms line by line. Using \eqref{eq_gradient_projection}, \eqref{eq:com_weight}, \eqref{eq:grad_weight} and \eqref{eq:lem_weight_1} we bound the first line by
\begin{equation*}
|\beta|^{\frac12} h^{\frac12} \|\nabla  (i_h-1) (v_h \varphi)\|_{\Omega}
+\|\mu^{\frac12} \nabla v_h \varphi\|_\Omega + 2 |\beta|^{\frac12} (h^\frac12 + \lambda^{-1}) \|v_h \varphi\|_\Omega
\le C \tnorm{v_h}_\varphi.
\end{equation*}
For the second line, using a global trace inequality and the stability of the projection we have
\begin{equation*}
(|\beta|+ \mu h^{-1})^{\frac12} \|(\pi_h - 1) (v_h \varphi)\|_{\partial \Omega}
\leq  C (|\beta|+ \mu h^{-1})^{\frac12} \|v_h \varphi\|_{\Omega}
\leq C |\beta|^{\frac12} \|v_h \varphi^{\frac12}\|_{\Omega}. 
\end{equation*}
Splitting the boundary into inflow and outflow and using \eqref{eq:bound_inflow}, we have 
$$
(|\beta|+ \mu h^{-1})^{\frac12} \|v_h \varphi\|_{\partial \Omega}
\leq C |\beta|^{\frac12} \|v_h \varphi\|_{\partial \Omega}
\leq C \|(v_h,0)\|_s + C\tnorm{v_h}_\varphi.
$$
For the contribution of the jump term, we insert $i_h$ and bound
\begin{align}
\begin{split}\label{eq:jump_projection}
s_\Omega(\pi_h (v_h \varphi), \pi_h (v_h \varphi))^{\frac12}
\leq{}& s_\Omega((\pi_h - i_h) (v_h \varphi), (\pi_h - i_h) (v_h\varphi))^{\frac12}\\
&+ s_\Omega((i_h-1) (v_h \varphi), (i_h - 1) (v_h \varphi))^{\frac12}\\
&+ s_\Omega(v_h \varphi, v_h \varphi)^{\frac12}.
\end{split}
\end{align}
We first observe that using \eqref{eq:trace_grad} and \eqref{eq:gradient_projection_interpolator}, we can bound the first term by
\begin{align*}
s_\Omega((\pi_h - i_h) (v_h \varphi), (\pi_h - i_h) (v_h\varphi))^{\frac12}
&\leq |\beta|^{\frac12} h^{\frac12} \| \nabla (\pi_h - i_h)  (v_h\varphi)\|_{\Omega} \\
&\leq |\beta|^{\frac12} h^{\frac12} \| \nabla (i_h - 1)  (v_h\varphi)\|_{\Omega}\\
&\leq C |\beta|^{\frac12} h^{\frac12} \|\nabla \varphi\|_{\infty,\Omega} \|v_h\|_\Omega\\
&\leq C |\beta|^{\frac12} (h^{\frac12}+\lambda^{-1}) \| v_h \varphi \|_\Omega,
\end{align*}
where for the last two inequalities we used the discrete commutator property \cref{lem:disc_com} together with the $\varphi$-bounds \eqref{eq:grad_weight} and \eqref{eq:lem_weight_1}.
Since $\varphi$ is Lipschitz continuous on $K$, $\varphi|_F$ is also Lipschitz continuous, and so $\varphi|_F \in W^{1,\infty}(F)$. The restriction of the nodal interpolant on $K$ onto $F$ gives the nodal interpolant on $F$, hence applying \cref{lem:disc_com} to $F$ instead of $K$ we have the discrete commutator estimate 
$$
h \|n \cdot \nabla (i_h-1) (v_h \varphi)\|_F
\leq C h |\varphi|_{W^{1,\infty}(K)} \|v_h\|_F
\leq C (h^{\frac12}+\lambda^{-1}) \| v_h \varphi\|_K,
$$
where in the last step we used \eqref{eq:grad_weight} and \eqref{eq:lem_weight_1} together with a discrete trace inequality.
After summation we have that
$$
s_\Omega((i_h-1) (v_h \varphi), (i_h - 1) (v_h \varphi))^{\frac12} \leq C \tnorm{v_h}_\varphi.
$$
Finally we use the trivial bound (since $|\varphi| < 1$)
$$
s_\Omega(v_h \varphi, v_h \varphi)^{\frac12} \leq s_\Omega(v_h, v_h)^{\frac12}.
$$
We conclude the proof by summing up the above contributions.
\end{proof}

We can now prove the following error estimate showing that, in the zone ${\mathring \omega_\beta}$ where we have stability, the convergence in the $L^2$-norm is of order $\mathcal{O}(h^{\frac32})$ on unstructured meshes, which is known to be optimal.
\begin{theorem}\label{thm:error_weighted_downstream}
Let $u \in H^2(\Omega)$ be the solution of \eqref{eq:intro_model_problem}
and $(u_h,z_h) \in [V_h]^2$ the solution to \eqref{eq:method_FEM}. Then there exists $h_0>0$ such that for all $h<h_0$ with $Pe(h) \gtrsim 1$ there holds
$$
\tnorm{u - u_h}_\varphi \leq C (|\beta|^{\frac12} h^{\frac32} |u|_{H^2(\Omega)} + |\beta|^{\frac12} h^{-\frac12} \|\delta\|_\omega ).
$$
\begin{proof}
Let $e_h = \pi_h u -  u_h \in V_h$, then $u - u_h = u- \pi_h u + e_h$.
By \cref{cor:weight_control_projection} there exists $\alpha>0$ such that
\begin{equation*}
\alpha \tnorm{e_h}_\varphi^2  \leq a_h(e_h, \pi_h (e_h\varphi) ) + C \gamma^{-1} \|(e_h,0)\|^2_s.
\end{equation*}
By Cauchy-Schwarz combined with \cref{lem:weight_control_lb} and Young's inequality
\begin{equation*}
-s_*(z_h, \pi_h (e_h\varphi)) \leq C \varepsilon_1^{-1} \|(0,z_h)\|^2_s + \varepsilon_1(\tnorm{e_h}_\varphi^2 + \|(e_h,0)\|^2_s),
\end{equation*}
for some $0 < \varepsilon_1 < \alpha/2$, hence
$$
\frac{\alpha}{2} \tnorm{e_h}_\varphi^2  \leq a_h(e_h, \pi_h (e_h\varphi))
+  s_*(z_h, \pi_h (e_h\varphi))  + C \varepsilon_1^{-1} \|(0,z_h)\|^2_s + C \|(e_h,0)\|^2_s.
$$
Applying the first equality of the consistency \cref{lem:consist} we obtain
\begin{equation}\label{eq:error_weight_stab}
\frac{\alpha}{2} \tnorm{e_h}_\varphi^2  \leq a_h(\pi_h u - u, \pi_h (e_h\varphi) )
+ C \varepsilon_1^{-1} \|(0,z_h)\|^2_s + C\|(e_h,0)\|^2_s.
\end{equation}
Since $\pi_h u - u \in V_h^\perp$ we may apply \cref{lem:cont} to bound
$$
a_h(\pi_h u - u, \pi_h (e_h\varphi) ) \leq C \|\pi_h u - u\|_\sharp \|(0,\pi_h (e_h\varphi))\|_s.
$$
From \cref{lem:weight_control_lb} and Young's inequality we thus have that for some $\varepsilon_2>0$,
\begin{equation*}
a_h(\pi_h u - u, \pi_h (e_h \varphi))
\leq  C ((1 + \varepsilon_2^{-1}) \|\pi_h u - u\|_\sharp^2 + \|(e_h,0)\|_s^2 + \varepsilon_2 \tnorm{e_h}_\varphi^2).
\end{equation*}
Taking $\varepsilon_2 < \alpha/4$ and combining the above bound with \eqref{eq:error_weight_stab} we see that
$$
\frac{\alpha}{4} \tnorm{e_h}_\varphi^2
\leq C((1+ \varepsilon_2^{-1}) \|\pi_h u - u\|_\sharp^2 + (1+\varepsilon_1^{-1})\|(e_h,z_h)\|_s^2).
$$
Since $\varepsilon_{1,2}$ are independent of $h$ we can absorb them in the generic constant $C$ and using the approximation inequality \eqref{eq:approx} together with \cref{prop:conv_stab}, we conclude that
\begin{align*}
\tnorm{e_h}_\varphi &\leq C (\mu^{\frac12} h + |\beta|^{\frac12} h^{\frac32}) (|u|_{H^2(\Omega)} + h^{-2} \|\delta\|_\omega) \\
&\leq C (|\beta|^{\frac12} h^{\frac32} |u|_{H^2(\Omega)} + |\beta|^{\frac12} h^{-\frac12} \|\delta\|_\omega ),
\end{align*}
where we used that $Pe(h)>1$.
\end{proof}
\end{theorem}

\subsection{Upstream estimates}\label{sec:error_estimates_upstream}
In this case we consider $\beta = (\beta_1,0)$ with $\beta_1 < 0$ and the data set $$\omega = (0,x)\times (y^{-}, y^{+})$$ touching part of the outflow boundary $\partial \Omega^+$. We must choose the weight function differently and this time we take a negative $\varphi$ given by \eqref{eq:weight_up}
\begin{equation*}
\varphi:= \psi_1 \psi_2 \in (-1,0).
\end{equation*}
It seems that in this case we can not simultaneously get control of the $L^2$-norm and the weighted $H^1$-norm and we have to sacrifice the latter since it is not uniform in $\mu$. We now take the weighted triple norm to be
\begin{equation}\label{eq:triple_norm_up}
\tnorm{v_h}_\varphi^2 := \| |\beta|^\frac12 v_h |\varphi|^{\frac12} \|_{\Omega}^2
+ \||\beta \cdot n|^{\frac12} v_h |\varphi|^{\frac12}\|_{ \partial \Omega^-}^2,
\end{equation}
and rederive the results obtained in \cref{sec:error_estimates_downstream}, aiming for a local error estimate. Since $\varphi \in (-1,0)$, we will use that $\| \cdot \varphi \|_\Omega \leq \| \cdot |\varphi|^\frac12 \|_\Omega$.

We start with an analogue of \cref{lem:weight_control} by taking $v_h \varphi$ as a test function in the weak bilinear form $a_h$ and notice that since $\varphi<0$ we now have that
$$
\frac12 \left( \| |\beta \cdot n|^{\frac12} v_h |\varphi|^{\frac12} \|_{\partial \Omega^-}^2  + |\beta| \|v_h |\varphi|^{\frac12}\|^2_\Omega \right)
=(\beta \cdot \nabla v_h, v_h \varphi)_\Omega + \frac12 \| |\beta \cdot n|^{\frac12} v_h |\varphi|^{\frac12} \|_{\partial \Omega^+}^2.
$$
Arguing as previously in \eqref{eq:bound_inflow} but now for the outflow boundary, we obtain the bound
\begin{align}
\begin{split}\label{eq:bound_outflow}
\| |\beta \cdot n|^{\frac12} v_h |\varphi|^{\frac12}\|_{\partial \Omega^+}
&\leq C |\beta|^{\frac12} (\| v_h |\varphi|^{\frac12}\|_{\partial \Omega^+ \cap \omega_\beta} + \| v_h |\varphi|^{\frac12}\|_{\partial \Omega^+ \setminus \omega_\beta})\\
&\leq C |\beta|^{\frac12} h^{-\frac12} \| v_h \|_{\omega} +  C |\beta|^{\frac12} h^{\frac32} \|v_h\|_{H^1(\Omega)}\\
&\leq C \gamma^{-\frac12} \|(v_h,0)\|_s,
\end{split}
\end{align}
and thus
\begin{equation}\label{eq:triple_conv_stab_up}
 \frac12 \left( |\beta| \|v_h|\varphi|^{\frac12}\|^2_\Omega
 + \| |\beta \cdot n|^{\frac12} v_h |\varphi|^{\frac12} \|^2_{\partial \Omega^-} \right)
 \leq  (\beta \cdot \nabla v_h, v_h
\varphi)_\Omega+ C \gamma^{-1} \|(v_h,0)\|^2_s.
\end{equation}
For the diffusion term we no longer have any positive contribution due to the change in sign of the weight function $\varphi$, since now
$$
(\mu \nabla v_h,\nabla (v_h \varphi))_{\Omega}
= - \|\mu^{\frac12} \nabla v_h |\varphi|^{\frac12}\|_{\Omega}^2 + (\mu \nabla v_h, v_h \nabla \varphi)_\Omega.
$$
We must therefore control this entirely using the stabilization.
Integrating by parts and using the weighted trace inequality \eqref{eq:lem_weight_trace}
\begin{align*}
(\mu \nabla v_h,\nabla (v_h \varphi))_{\Omega} - \left< \mu \nabla v_h \cdot n , v_h \varphi\right>_{\partial \Omega}
&= \sum_{F \in \mathcal{F}_i} \int_{F} \mu \jump{\nabla v_h \cdot n} v_h \varphi ~\mbox{d}s \\
& \leq C \gamma^{-\frac12} s_\Omega(v_h,v_h)^{\frac12} \mu^\frac12 h^{-1} \|v_h \varphi\|_\Omega \\
& \leq C \gamma^{-\frac12} s_\Omega(v_h,v_h)^{\frac12} \mu^\frac12 h^{-1} \|v_h |\varphi|^{\frac12}\|_\Omega.
\end{align*}
To bound this by the triple norm we can simply use that $|\varphi| < 1$ and $\mu \le |\beta| h$, giving that $\mu^\frac12 h^{-1} |\varphi|^{\frac12} \leq |\beta|^\frac12 h^{-\frac12}.$
Hence we have that for some $\varepsilon >0$,
$$
|(\mu \nabla v_h,\nabla (v_h \varphi))_{\Omega} - \left< \mu \nabla v_h \cdot n , v_h \varphi\right>_{\partial \Omega}|
\leq C \varepsilon^{-1} \gamma^{-1} h^{-1} s_\Omega(v_h,v_h) + C \varepsilon \tnorm{v_h}_\varphi^2.
$$
However, when $Pe(h) h > 1$ one can obtain a better estimate due to $\mu^\frac12 h^{-1} |\varphi|^{\frac12} \leq |\beta|^\frac12,$
which gives that	
$$
|(\mu \nabla v_h,\nabla (v_h \varphi))_{\Omega} - \left< \mu \nabla v_h \cdot n , v_h \varphi\right>_{\partial \Omega}|
\leq C \varepsilon^{-1} \gamma^{-1} s_\Omega(v_h,v_h) + C \varepsilon \tnorm{v_h}_\varphi^2.
$$
Summing these contributions we obtain the following result corresponding to \cref{lem:weight_control}.

\begin{lemma}\label{lem:weight_control_upstream}
	There exists $\alpha>0$ such that for all $v_h \in V_h$ we have
	$$
	\alpha \tnorm{v_h}_\varphi^2 \leq  a_h(v_h, v_h \varphi) + C h^{-1} \|(v_h,0)\|^2_s,\text{ when } 1 \lesssim Pe(h) < h^{-1},
	$$
	and
	$$
	\alpha \tnorm{v_h}_\varphi^2 \leq  a_h(v_h, v_h \varphi) + C \|(v_h,0)\|^2_s,\text{ when } Pe(h) > h^{-1}.
	$$
\end{lemma}
Again, we can refine the control over the triple norm $\tnorm{v_h}_\varphi$ by taking the projection $\pi_h (v_h \varphi) \in V_h$ as a test function and we obtain corresponding results.
\begin{corollary}\label{cor:weight_control_projection_upstream}
	There exists $\alpha>0$ such that for all $v_h \in V_h$ we have
	$$
	\alpha \tnorm{v_h}_\varphi^2 \leq  a_h(v_h, \pi_h(v_h \varphi)) + C h^{-1} \|(v_h,0)\|^2_s,\text{ when } 1 \lesssim Pe(h) < h^{-1},
	$$
	and
	$$
	\alpha \tnorm{v_h}_\varphi^2 \leq  a_h(v_h, \pi_h(v_h \varphi)) + C  \|(v_h,0)\|^2_s,\text{ when } Pe(h) > h^{-1}.
	$$
\begin{proof}
	The argument in the proof of \cref{cor:weight_control_projection} remains valid with the remark that we now use the inequality $|\varphi| < |\varphi|^\frac12$.
\end{proof}
\end{corollary}
\begin{lemma}\label{lem:weight_control_lb_upstream}
	For all $v_h \in V_h$ there holds
	$$
	\|(0,\pi_h (v_h \varphi))\|^2_s \leq C (h^{-1} \tnorm{v_h}^2_\varphi + \|(v_h,0)\|^2_s),\text{ when } 1 \lesssim Pe(h) < h^{-1},
	$$
	and
	$$
	\|(0,\pi_h (v_h \varphi))\|^2_s \leq C (\tnorm{v_h}^2_\varphi + \|(v_h,0)\|^2_s),\text{ when } Pe(h) > h^{-1}.
	$$
	\begin{proof}
	We follow the proof of \cref{lem:weight_control_lb}	and we focus on the bounds that are now different. As before, by the triangle inequality we have that up to a constant
	\begin{align*}
	\|(0,\pi_h (v_h \varphi))\|_s \leq{}& \|\mu^{\frac12} \nabla (\pi_h - 1) (v_h	\varphi))\|_\Omega + \|\mu^{\frac12} v_h \nabla \varphi\|_\Omega + \|\mu^{\frac12} \nabla v_h \varphi \|_\Omega \\
	&+ (|\beta|+ \mu h^{-1})^{\frac12} (\|(\pi_h-1) (v_h \varphi)\|_{\partial \Omega}
	+ \|v_h\varphi\|_{\partial \Omega} )\\
	&+ s_\Omega(\pi_h (v_h \varphi), \pi_h (v_h \varphi))^{\frac12}.
	\end{align*}
	The first two terms can be bounded by $C \tnorm{v_h}_\varphi$ as previously. For the third one, we can use the inverse inequality \eqref{eq:inverse_ineq} and \eqref{eq:lem_weight_1} to obtain
	$$
	\|\mu^{\frac12} \nabla v_h \varphi \|_\Omega \leq C \mu^{\frac12} h^{-1} \|\varphi\|_{\infty,\Omega} \|v_h \|_\Omega
	\leq C \mu^{\frac12} h^{-1} \|v_h \varphi \|_\Omega
	\leq C \mu^{\frac12} h^{-1} \|v_h |\varphi|^\frac12 \|_\Omega.
	$$
	Hence we have that
	$$\|\mu^{\frac12} \nabla v_h \varphi \|_\Omega \leq C h^{-\frac12} \tnorm{v_h}_\varphi, \text{ when } 1 \lesssim Pe(h) < h^{-1},$$
	and
	$$\|\mu^{\frac12} \nabla v_h \varphi \|_\Omega \leq C \tnorm{v_h}_\varphi, \text{ when } Pe(h)>h^{-1}.$$
	Arguing as previously, we can bound the second line by $C \tnorm{v_h}_\varphi$ using \eqref{eq:bound_outflow} instead of \eqref{eq:bound_inflow}. We conclude the proof by recalling the estimate \eqref{eq:jump_projection} for the jump term and the subsequent bounds.
	\end{proof}
\end{lemma}

We now prove the weighted error estimate in the upstream case showing that in the stability region $\mathring{\omega}_\beta$ we have quasi-optimal convergence for high P\'eclet numbers and a reduction of the convergence order by $\mathcal{O}(h^\frac12)$ in an intermediate regime.
\begin{theorem}\label{thm:error_weighted_upstream}
	Let $u \in H^2(\Omega)$ be the solution of \eqref{eq:intro_model_problem}
	and $(u_h,z_h) \in [V_h]^2$ the solution to \eqref{eq:method_FEM}, then there holds
	$$
	\tnorm{u - u_h}_\varphi \leq C (|\beta|^{\frac12} h |u|_{H^2(\Omega)} + |\beta|^{\frac12} h^{-1} \|\delta\|_\omega ), \text{ when } 1 \lesssim Pe(h) < h^{-1},
	$$
	and
	$$
	\tnorm{u - u_h}_\varphi \leq C (|\beta|^{\frac12} h^{\frac32} |u|_{H^2(\Omega)} + |\beta|^{\frac12} h^{-\frac12} \|\delta\|_\omega ), \text{ when } Pe(h) > h^{-1}.
	$$
	\begin{proof}
	We combine \cref{lem:weight_control_upstream}, \cref{cor:weight_control_projection_upstream} and \cref{lem:weight_control_lb_upstream} as in the proof of \cref{thm:error_weighted_downstream} and note that the argument holds verbatim when $Pe(h) > h^{-1}$. Observe that when $1 \lesssim Pe(h) < h^{-1}$ we similarly obtain
	for some $\alpha>0$ and $0 < \varepsilon_1 < \alpha/2$,
	\begin{equation}\label{eq:error_weight_stab_upstream}
	\frac{\alpha}{2} \tnorm{e_h}_\varphi^2  \leq a_h(\pi_h u - u, \pi_h (e_h\varphi) )
	+ C \varepsilon_1^{-1} \|(0,z_h)\|^2_s + C h^{-1} \|(e_h,0)\|^2_s.
	\end{equation}
	Since $\pi_h u - u \in V_h^\perp$ we may apply \cref{lem:cont} to bound
	$$
	a_h(\pi_h u - u, \pi_h (e_h\varphi) ) \leq C \|\pi_h u - u\|_\sharp \|(0,\pi_h (e_h\varphi))\|_s.
	$$
	From \cref{lem:weight_control_lb_upstream} and Young's inequality we thus have that for some $\varepsilon_2>0$,
	\begin{equation*}
	a_h(\pi_h u - u, \pi_h (e_h \varphi))
	\leq  C ((1 + \varepsilon_2^{-1} h^{-1}) \|\pi_h u - u\|_\sharp^2 + \|(e_h,0)\|_s^2 + \varepsilon_2 \tnorm{e_h}_\varphi^2).
	\end{equation*}
	Taking $\varepsilon_2 < \alpha/4$ and combining the above bound with \eqref{eq:error_weight_stab_upstream} we see that
	$$
	\frac{\alpha}{4} \tnorm{e_h}_\varphi^2
	\leq C((1+ \varepsilon_2^{-1} h^{-1} ) \|\pi_h u - u\|_\sharp^2 + \varepsilon_1^{-1} h^{-1} \|(e_h,z_h)\|_s^2).
	$$
	Since $\varepsilon_{1,2}$ are independent of $h$ we can absorb them in the constant $C$ and conclude the proof by using the approximation inequality \eqref{eq:approx} and \cref{prop:conv_stab} to obtain that
	\begin{align*}
	\tnorm{e_h}_\varphi &\leq C (\mu^{\frac12} h^{\frac12} + |\beta|^{\frac12} h) (|u|_{H^2(\Omega)} + h^{-2} \|\delta\|_\omega) \\
	&\leq C (|\beta|^{\frac12} h |u|_{H^2(\Omega)} + |\beta|^{\frac12} h^{-1} \|\delta\|_\omega ).
	\end{align*}
	\end{proof}
\end{theorem}
The error bounds in this section contain a global Sobolev norm.
This may be large in the presence of layers and it would be optimal to replace it with a local regularity measure.
However, it is not clear how to reconcile this goal with the ill-posed character of the problem.
Inserting everywhere the weight function as in the well-posed case \cite{BGL09}, would perturb the stability of the optimality system and, due to the lack of physical coercivity, the residual terms would then be hard to control.
One could also consider a reduced transport equation (with zero diffusivity)
and define the far field, where the solution is unknown, by a smooth extension, but is not obvious how this could be constructed in our context, and without requiring regularity for the right-hand side.
Nonetheless, in numerical experiments we observe a good regularity behaviour, i.e. layers do not pollute the solution in the stability zone, as shown in \cref{fig:layer_diffusion,fig:layer_convection}.
\section{Numerical experiments}
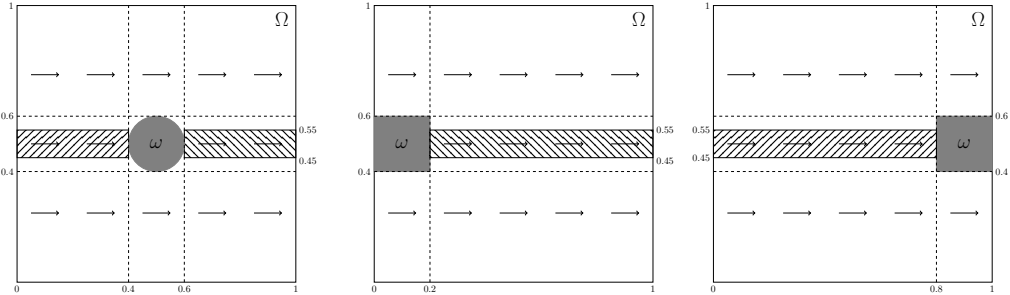
\begin{figure}[H]
	\begin{subfigure}{0.32\textwidth}
		\resizebox{\textwidth}{!}{
			\centering
			\begin{tikzpicture}
				\draw (0,0) rectangle (10,10);
				\draw (0,0) node[align=center, below] {0};
				\draw (4,0) node[align=center, below] {0.4};
				\draw (6,0) node[align=center, below] {0.6};
				\draw (10,0) node[align=center, below] {1};
				\draw (0,10) node[align=center, left] {1};
				\draw (0,6) node[align=center, left] {0.6};
				\draw (0,4) node[align=center, left] {0.4};
				\draw (10,5.5) node[align=center, right] {0.55};
				\draw (10,4.4) node[align=center, right] {0.45};
				
				\filldraw[pattern=north east lines] (0,4.5) rectangle (4,5.5);
				\filldraw[pattern=north west lines] (6,4.5) rectangle (10,5.5);
				
				\filldraw[color=gray] (5,5) circle (1);
				\draw (5,5) node[align=center] {\Huge $\omega$};
				\draw (9.5,9.5) node[align=center] {\Huge $\Omega$};
				
				\draw[dashed] (0,4) -- (10,4);
				\draw[dashed] (0,6) -- (10,6);
				
				\draw[dashed] (4,0) -- (4,10);
				\draw[dashed] (6,0) -- (6,10);
				
				\foreach \x in {0.5, 2.5, 4.5, 6.5, 8.5} {
					\draw[->,thick] (\x,2.5) -- +(1,0);
					\draw[->,thick] (\x,7.5) -- ++(1,0);
				}
				\foreach \x in {0.5, 2.5, 6.5, 8.5} {
					\draw[->,thick] (\x,5) -- +(1,0);
				}
		\end{tikzpicture}}
		\caption{$\omega = B((0.5,0.5),0.1)$, measuring the error downstream and upstream.}
		\label{fig:data_sets_disk}
	\end{subfigure}
	\hfill
	\begin{subfigure}{0.32\textwidth}
		\resizebox{\textwidth}{!}{
			\begin{tikzpicture}
				\draw (0,0) rectangle (10,10);
				\draw (0,0) node[align=center, below] {0};
				\draw (2,0) node[align=center, below] {0.2};
				\draw (10,0) node[align=center, below] {1};
				\draw (0,10) node[align=center, left] {1};
				\draw (0,6) node[align=center, left] {0.6};
				\draw (0,4) node[align=center, left] {0.4};
				\draw (10,5.5) node[align=center, right] {0.55};
				\draw (10,4.4) node[align=center, right] {0.45};
				
				\filldraw[color=gray] (0,4) rectangle (2,6);
				\filldraw[pattern=north west lines] (2,4.5) rectangle (10,5.5);
				
				\draw (1,5) node[align=center] {\Huge $\omega$};
				\draw (9.5,9.5) node[align=center] {\Huge $\Omega$};
				
				\draw[dashed] (0,4) -- (10,4);
				\draw[dashed] (0,6) -- (10,6);
				
				\draw[dashed] (2,0) -- (2,10);
				
				\foreach \x in {0.5, 2.5, 4.5, 6.5, 8.5} {
					\draw[->,thick] (\x,2.5) -- +(1,0);
					\draw[->,thick] (\x,7.5) -- ++(1,0);
				}
				\foreach \x in {2.5, 4.5, 6.5, 8.5} {
					\draw[->,thick] (\x,5) -- +(1,0);
				}
		\end{tikzpicture}}
		\caption{$\omega = (0,0.2)\times(0.4,0.6)$, measuring the error downstream in $(0.2,1)\times (0.45,0.55)$.}
		\label{fig:data_sets_side_down}
	\end{subfigure}
	\begin{subfigure}{0.32\textwidth}
		\resizebox{\textwidth}{!}{
			\begin{tikzpicture}
				\draw (0,0) rectangle (10,10);
				\draw (0,0) node[align=center, below] {0};
				\draw (8,0) node[align=center, below] {0.8};
				\draw (10,0) node[align=center, below] {1};
				\draw (0,10) node[align=center, left] {1};
				\draw (0,5.5) node[align=center, left] {0.55};
				\draw (0,4.5) node[align=center, left] {0.45};
				\draw (10,6) node[align=center, right] {0.6};
				\draw (10,4) node[align=center, right] {0.4};
				
				\filldraw[color=gray] (8,4) rectangle (10,6);
				\filldraw[pattern=north east lines] (0,4.5) rectangle (8,5.5);
				
				\draw (9,5) node[align=center] {\Huge $\omega$};
				\draw (9.5,9.5) node[align=center] {\Huge $\Omega$};
				
				\draw[dashed] (0,4) -- (10,4);
				\draw[dashed] (0,6) -- (10,6);
				
				\draw[dashed] (8,0) -- (8,10);
				
				\foreach \x in {0.5, 2.5, 4.5, 6.5, 8.5} {
					\draw[->,thick] (\x,2.5) -- +(1,0);
					\draw[->,thick] (\x,7.5) -- ++(1,0);
				}
				\foreach \x in {0.5, 2.5, 4.5, 6.5} {
					\draw[->,thick] (\x,5) -- +(1,0);
				}
		\end{tikzpicture}}
		\caption{$\omega = (0.8,1)\times(0.4,0.6)$, measuring the error upstream in $(0,0.8)\times (0.45,0.55)$.}
		\label{fig:data_sets_side_up}
	\end{subfigure}
	\caption{Data set $\omega$ and error measurement regions (hatched).}
	\label{fig:data_sets}
\end{figure}
We let $\Omega$ be the unit square and illustrate the performance of the numerical method \eqref{eq:method_FEM} for different locations of the data domain $\omega$ and different regions of interest where we measure the approximation error. The computational domains are given in \cref{fig:data_sets} and the implementation is done using FEniCS \cite{fenics}. In all the examples below we have used uniform triangulations with alternating left/right diagonals.
In the definition of $s_\Omega$ and $s_*$ we have taken the stabilization parameters $\gamma = 10^{-5}$ and $\gamma_* = 1$, and $\zeta = 2$ for $s_\omega$.
The effect of different combinations of $\gamma$ and $\gamma_*$ on the $L^2$-errors is shown in \cref{fig:centered_disc_params} and \cref{fig:centered_disc_params_global} when data is given in a centered disk. Similar results are obtained when the data set is near the inflow/outflow boundary. Notice that our choice is empirically close to being optimal both locally and globally.

\begin{figure}[h]
	\begin{subfigure}{0.48\textwidth}
		\includegraphics[draft=false, width=\textwidth]{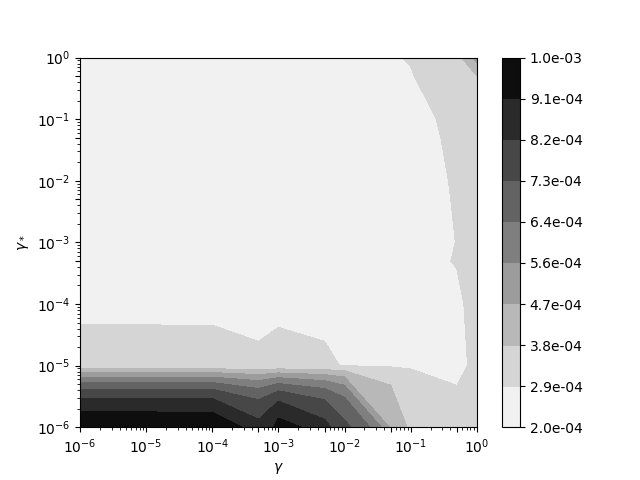}
		\caption{$\mu = 10^{-3}$.}
	\end{subfigure}
	\hfill
	\begin{subfigure}{0.48\textwidth}
		\includegraphics[draft=false, width=\textwidth]{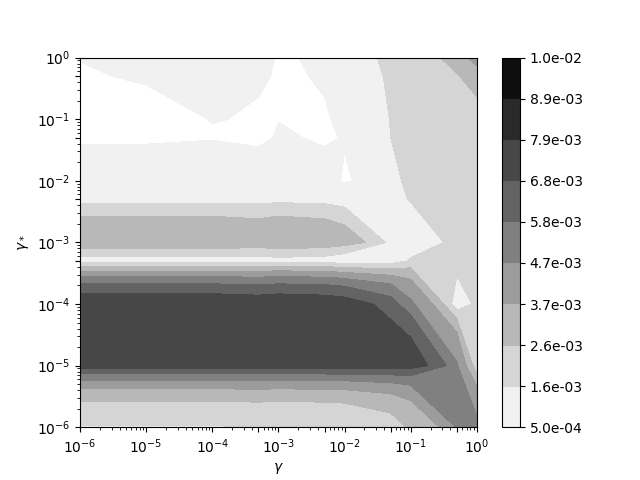}
		\caption{$\mu = 10^{-2}$.}
	\end{subfigure}
	\caption{Varying the stabilization parameters $\gamma$ and $\gamma_*$. Absolute $L^2$-errors downstream, computational domains in \cref{fig:data_sets_disk}. $\beta = (1,0)$, exact solution $u = 2\sin(5\pi x)\sin (5\pi y)$. Similar results in the upstream case.}
	\label{fig:centered_disc_params}
\end{figure}

\begin{figure}[h]
	\begin{subfigure}{0.48\textwidth}
		\includegraphics[draft=false, width=\textwidth]{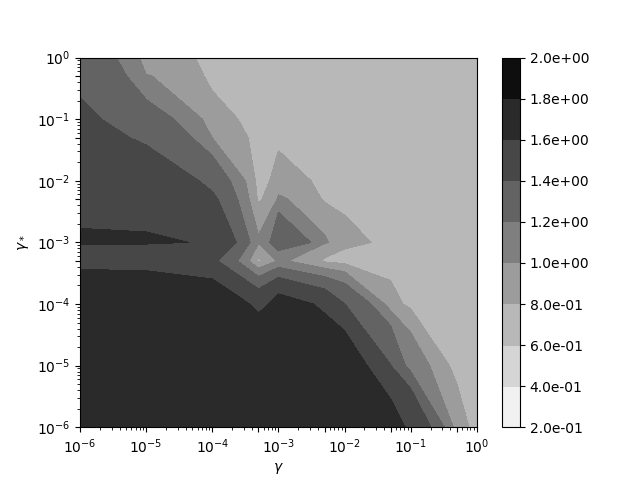}
		\caption{$\mu = 10^{-3}$.}
	\end{subfigure}
	\hfill
	\begin{subfigure}{0.48\textwidth}
		\includegraphics[draft=false, width=\textwidth]{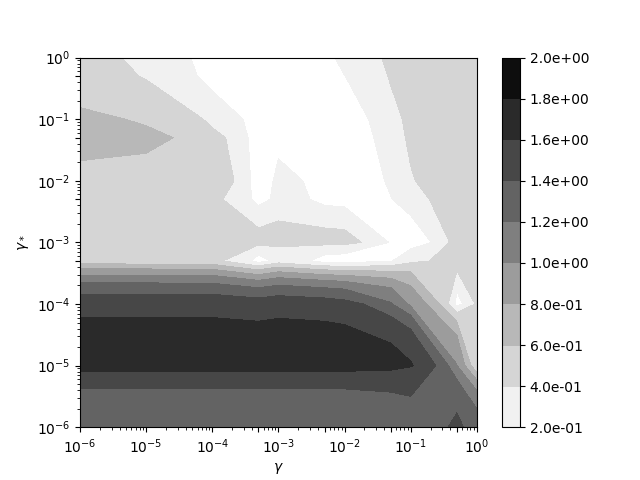}
		\caption{$\mu = 10^{-2}$.}
	\end{subfigure}
	\caption{Varying the stabilization parameters $\gamma$ and $\gamma_*$. Absolute $L^2$-errors globally, computational domains in \cref{fig:data_sets_disk}. $\beta = (1,0)$, exact solution $u = 2\sin(5\pi x)\sin (5\pi y)$.}
	\label{fig:centered_disc_params_global}
\end{figure}

We first show convergence plots both downstream and upstream from the data set when varying the diffusion coefficient $\mu$ and keeping the convection field $\beta$ fixed. As in the case of well-posed convection-dominated problems, the observed $L^2$-convergence order is typically $\mathcal{O}(h^2)$, surpassing by $\mathcal{O}(h^\frac12)$ the weighted error estimates proven for general meshes (see \cref{fig:sides_L2}, for example).
\begin{figure}[h]
	\begin{subfigure}{0.48\textwidth}
		\includegraphics[draft=false, width=\textwidth]{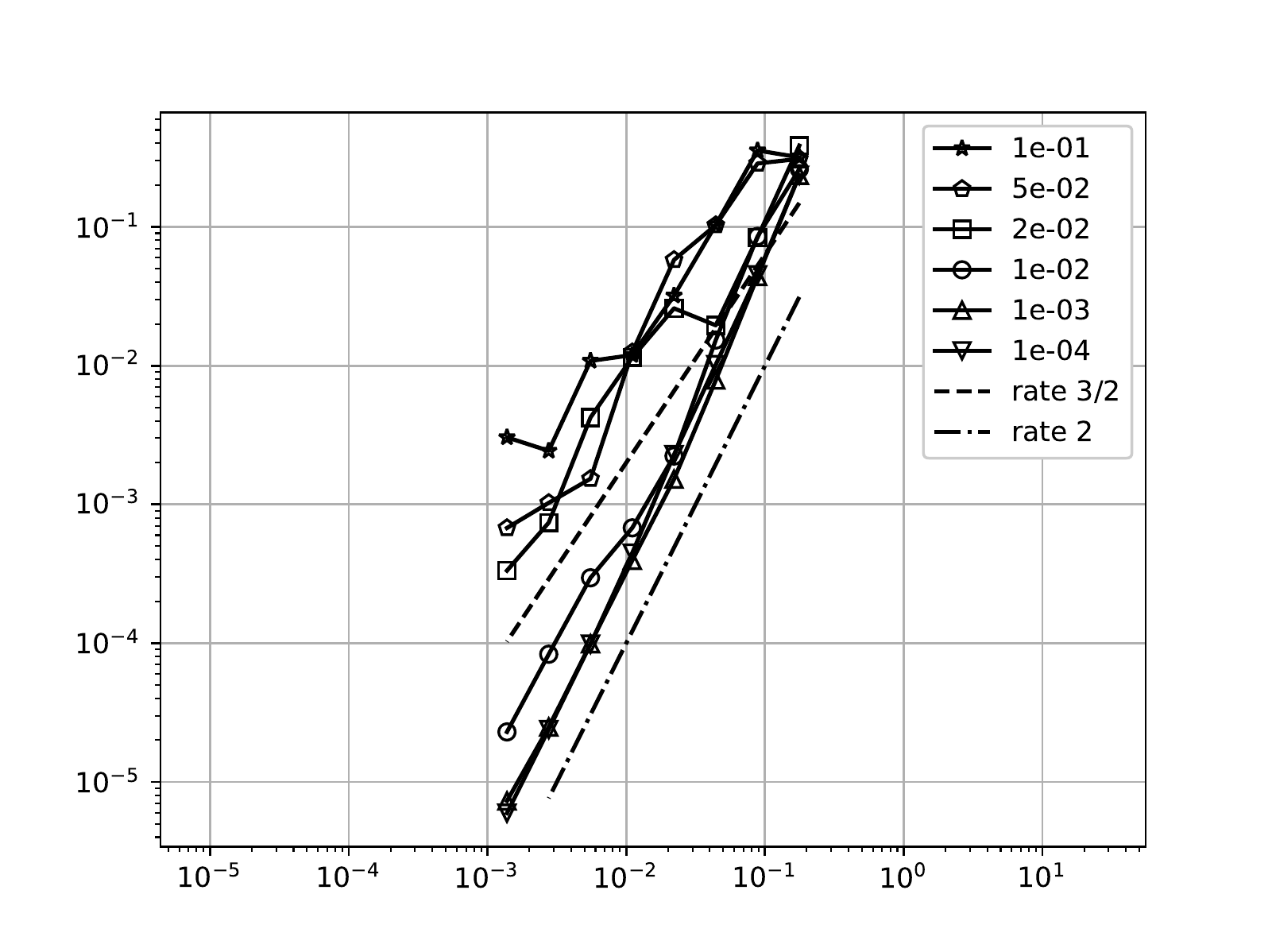}
		\caption{Computational domains in \cref{fig:data_sets_side_down}.}
		\label{fig:leftside_downstream}
	\end{subfigure}
	\hfill
	\begin{subfigure}{0.48\textwidth}
		\includegraphics[draft=false, width=\textwidth]{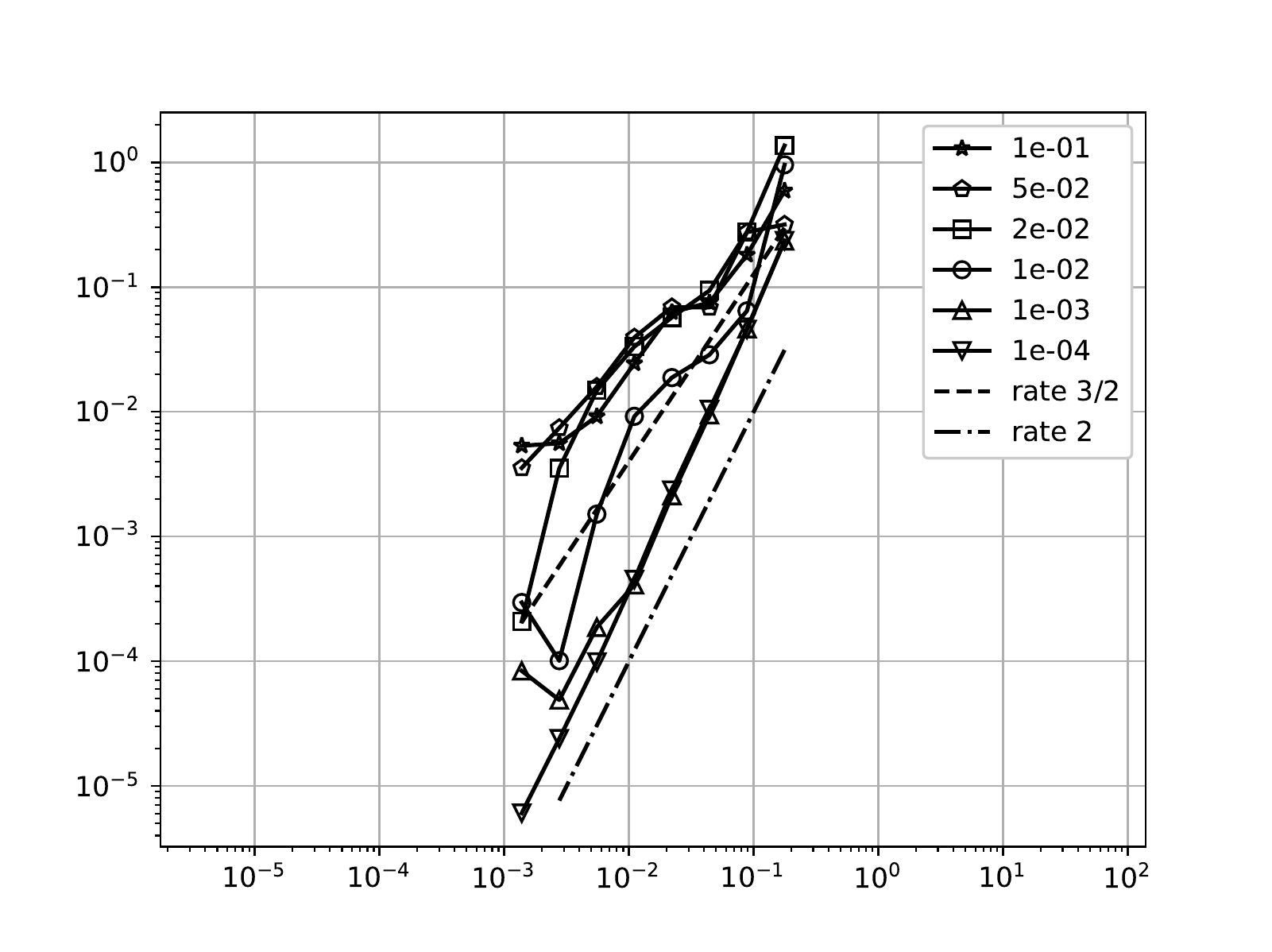}
		\caption{Computational domains in \cref{fig:data_sets_side_up}.}
		\label{fig:rightside_upstream}
	\end{subfigure}
	\caption{Absolute $L^2$-errors against mesh size $h$ when varying the diffusion coefficient $\mu$ for fixed $\beta = (1,0)$, exact solution $u = 2\sin(5\pi x)\sin (5\pi y)$.}
	\label{fig:sides_L2}
\end{figure}
\begin{figure}[h]
	\begin{subfigure}{0.48\textwidth}
		\includegraphics[draft=false, width=\textwidth]{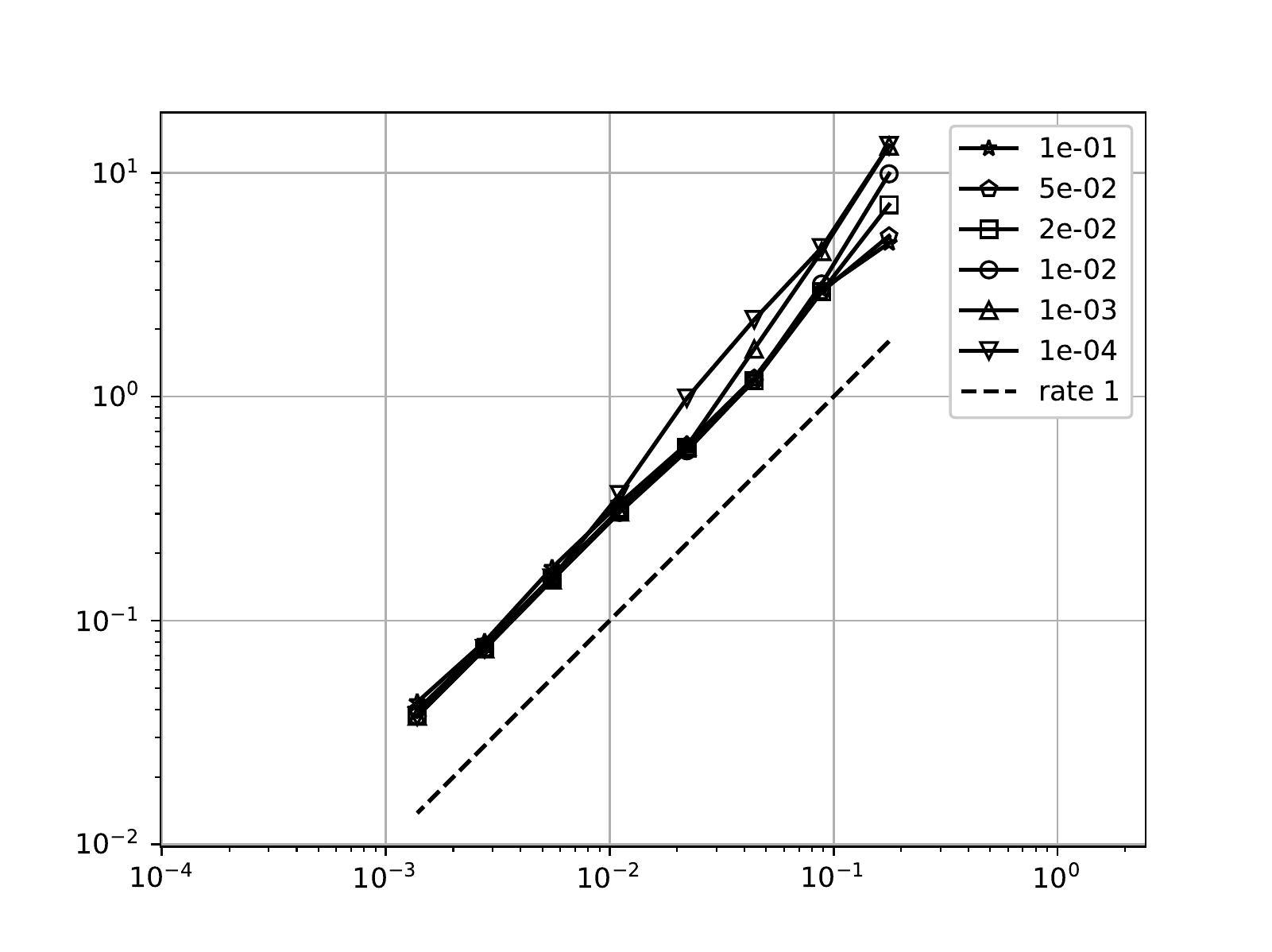}
		\caption{Computational domains in \cref{fig:data_sets_side_down}.}
		\label{fig:leftside_downstream}
	\end{subfigure}
	\hfill
	\begin{subfigure}{0.48\textwidth}
		\includegraphics[draft=false, width=\textwidth]{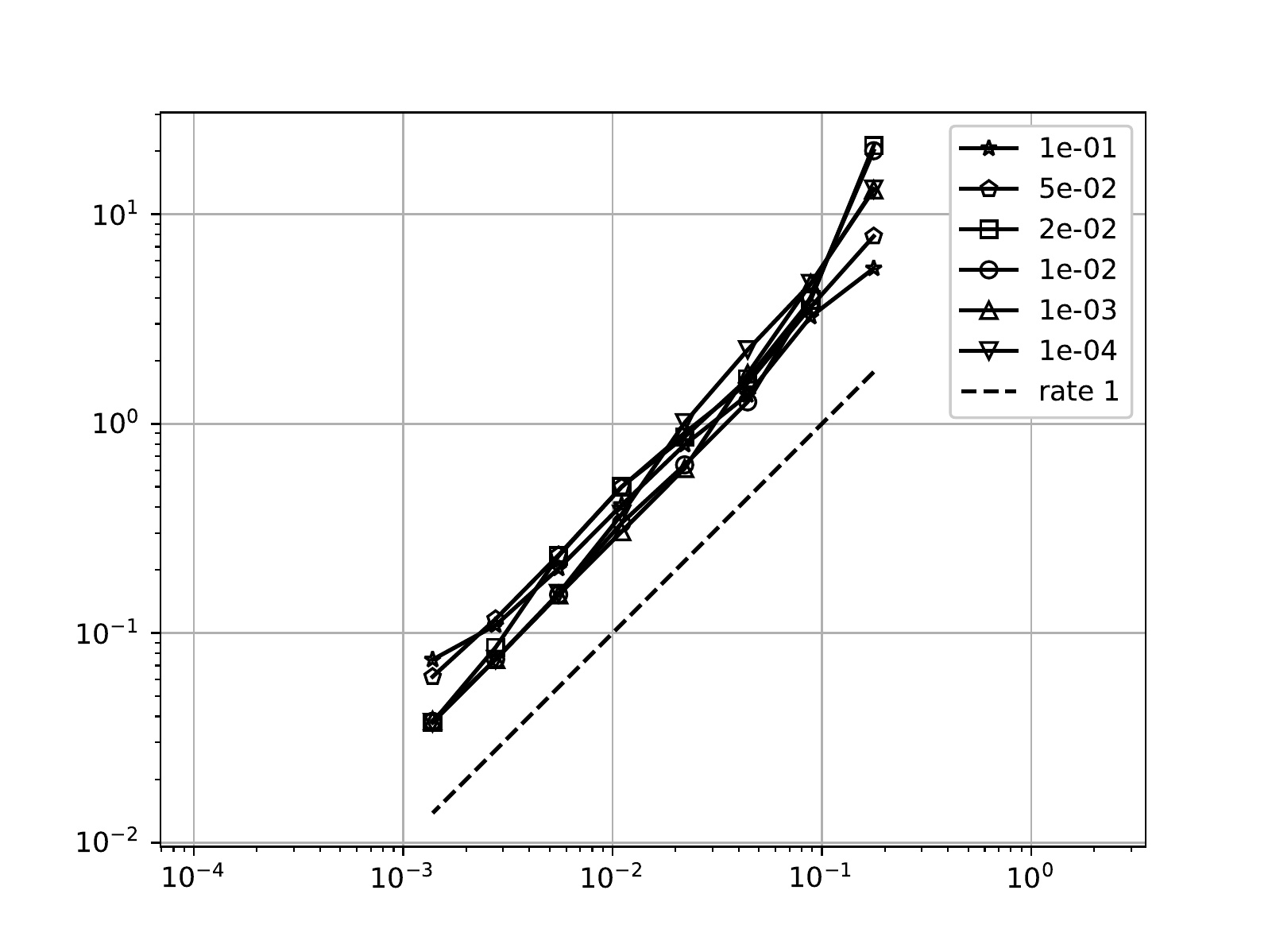}
		\caption{Computational domains in \cref{fig:data_sets_side_up}.}
		\label{fig:rightside_upstream}
	\end{subfigure}
	\caption{Absolute $H^1$-errors against mesh size $h$ when varying the diffusion coefficient $\mu$ for fixed $\beta = (1,0)$, exact solution $u = 2\sin(5\pi x)\sin (5\pi y)$.}
	\label{fig:sides_H1}
\end{figure}

\subsection{Data set near the inflow/outflow boundary} We consider the data set $\omega$ near the inflow and outflow boundaries of $\Omega$, as assumed in \cref{section_stab_weight}.
We observe in \cref{fig:sides_L2} that as diffusion is reduced the convergence order for the $L^2$-errors increases, culminating with quadratic convergence when convection dominates. Confirming the theoretical analysis in \cref{sec:error_estimates_upstream}, we note the presence of an intermediate regime for P\'eclet numbers in which the upstream convergence orders are reduced and the upstream errors are typically larger. This can also be seen in \cref{fig:down_vs_up} where we consider the diffusion coefficient $\mu = 10^{-2}$ and an interior data set. The errors in the $H^1$-seminorm are given in \cref{fig:sides_H1} which shows almost linear convergence. This is probably due to the small contribution of the gradient term in the triple norm \eqref{eq:triple_norm_down}.
\begin{figure}[h]
	\begin{subfigure}{0.48\textwidth}
		\includegraphics[draft=false, width=\textwidth]{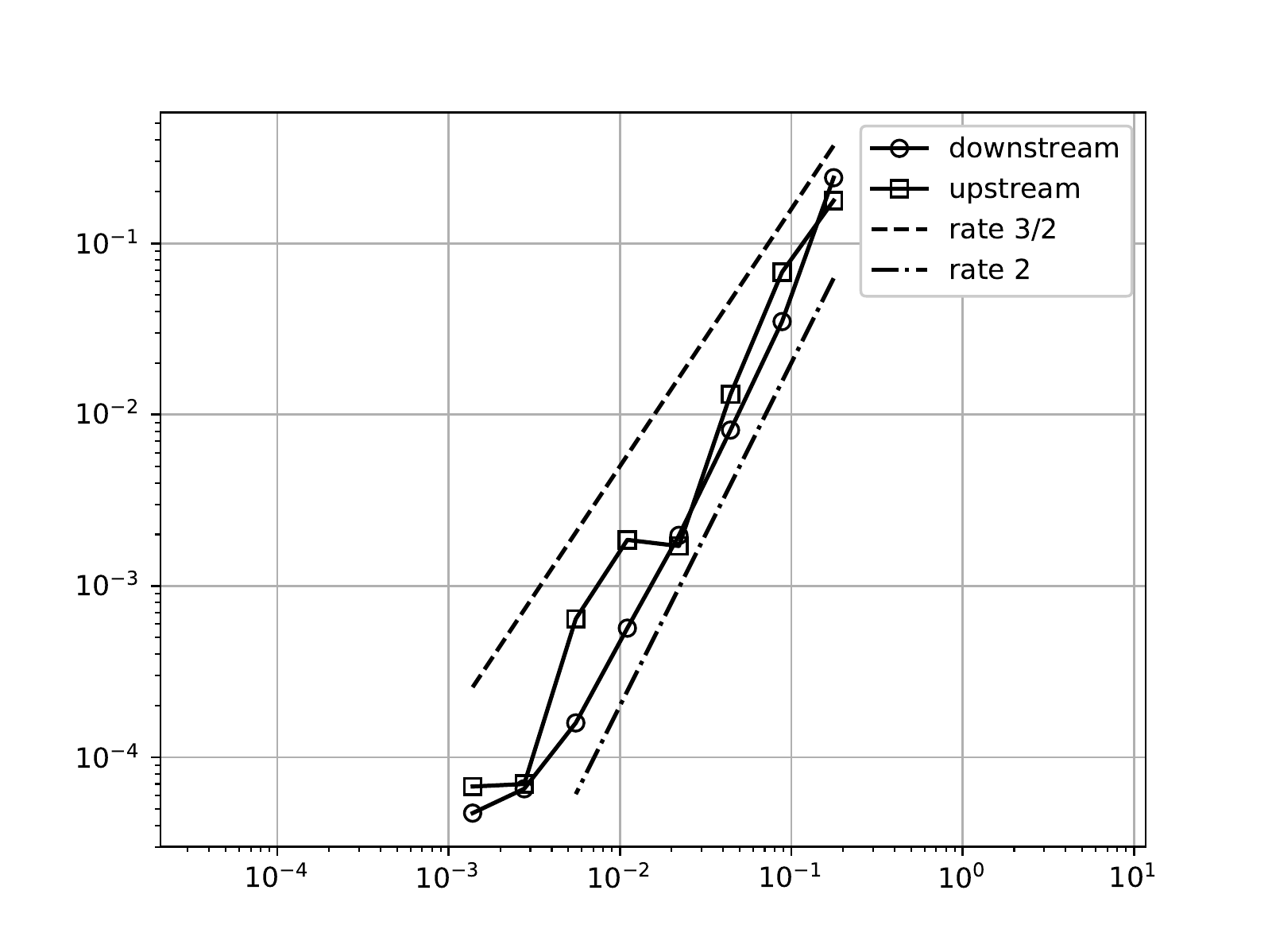}
		\caption{Domains in \cref{fig:data_sets_disk}.}
	\end{subfigure}
	\hfill
	\begin{subfigure}{0.48\textwidth}
		\includegraphics[draft=false, width=\textwidth]{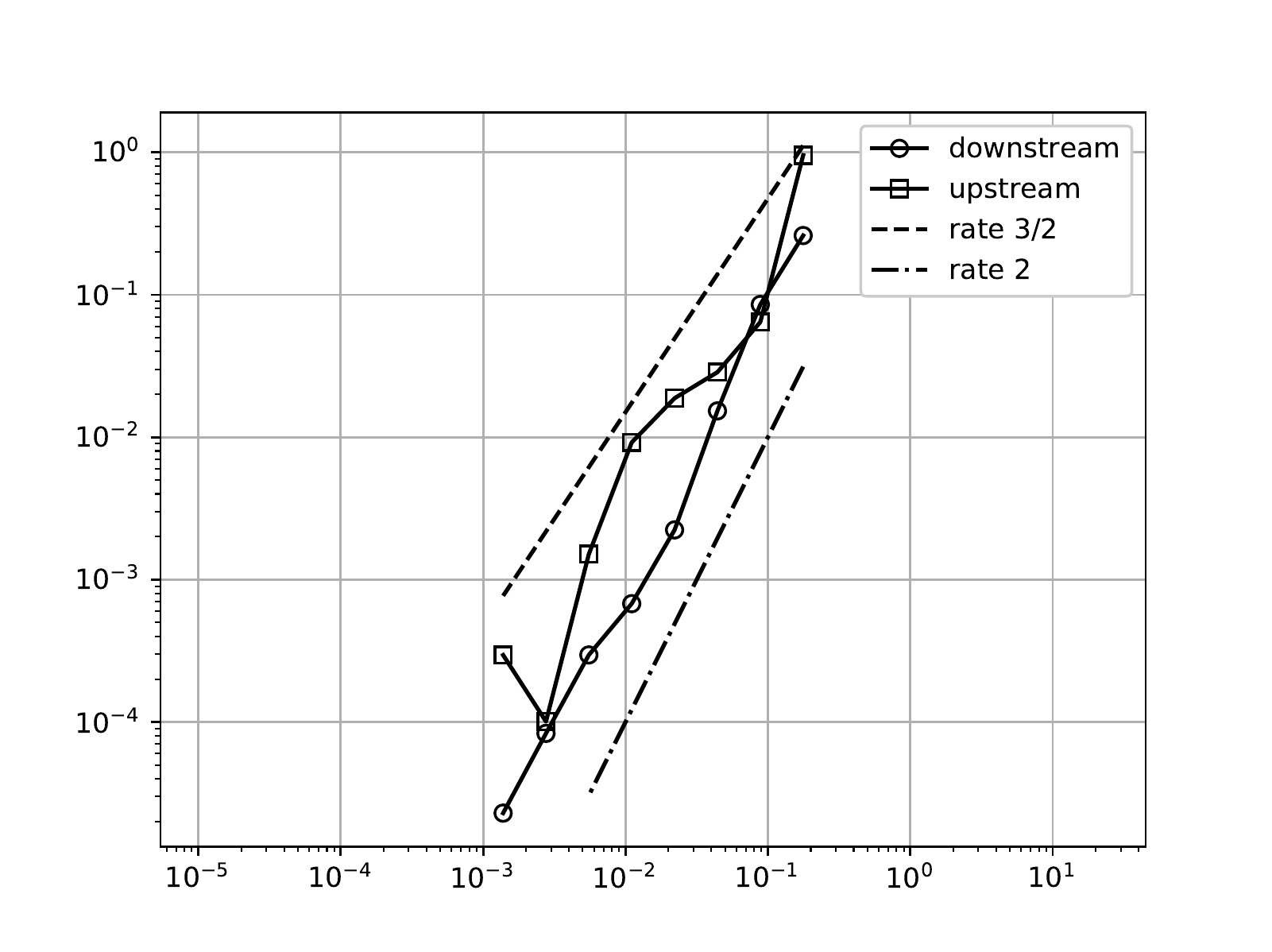}
		\caption{Domains in \cref{fig:data_sets_side_down} and \cref{fig:data_sets_side_up}.}
	\end{subfigure}
	\caption{Absolute $L^2$-errors against mesh size $h$, downstream vs upstream for $\mu = 10^{-2}$, $\beta = (1,0)$, exact solution $u = 2\sin(5\pi x)\sin (5\pi y)$.}
	\label{fig:down_vs_up}
\end{figure}

\begin{figure}[h]
	\begin{subfigure}{0.48\textwidth}
		\includegraphics[draft=false, width=\textwidth]{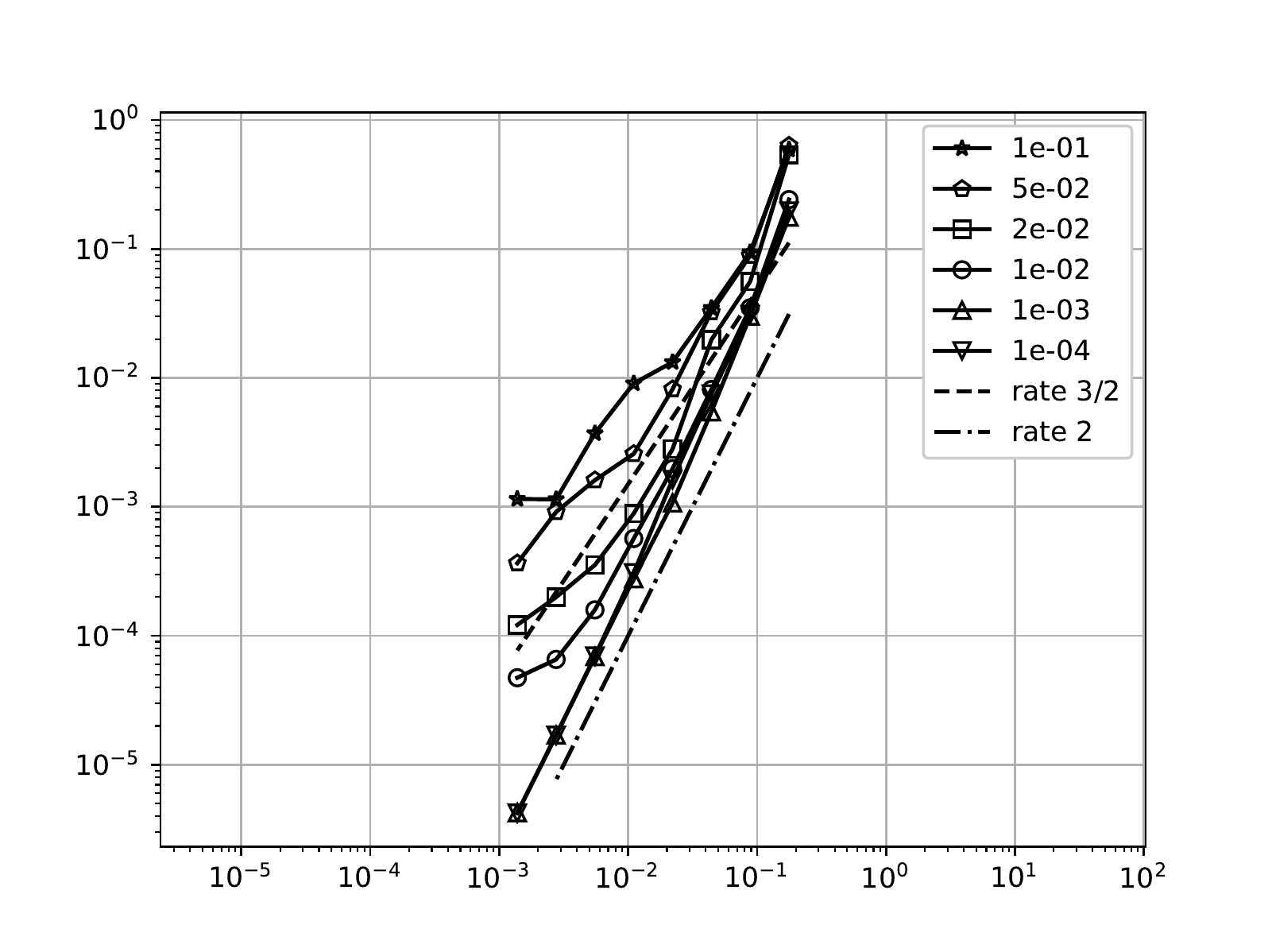}
		\caption{Downstream.}
	\end{subfigure}
	\hfill
	\begin{subfigure}{0.48\textwidth}
		\includegraphics[draft=false, width=\textwidth]{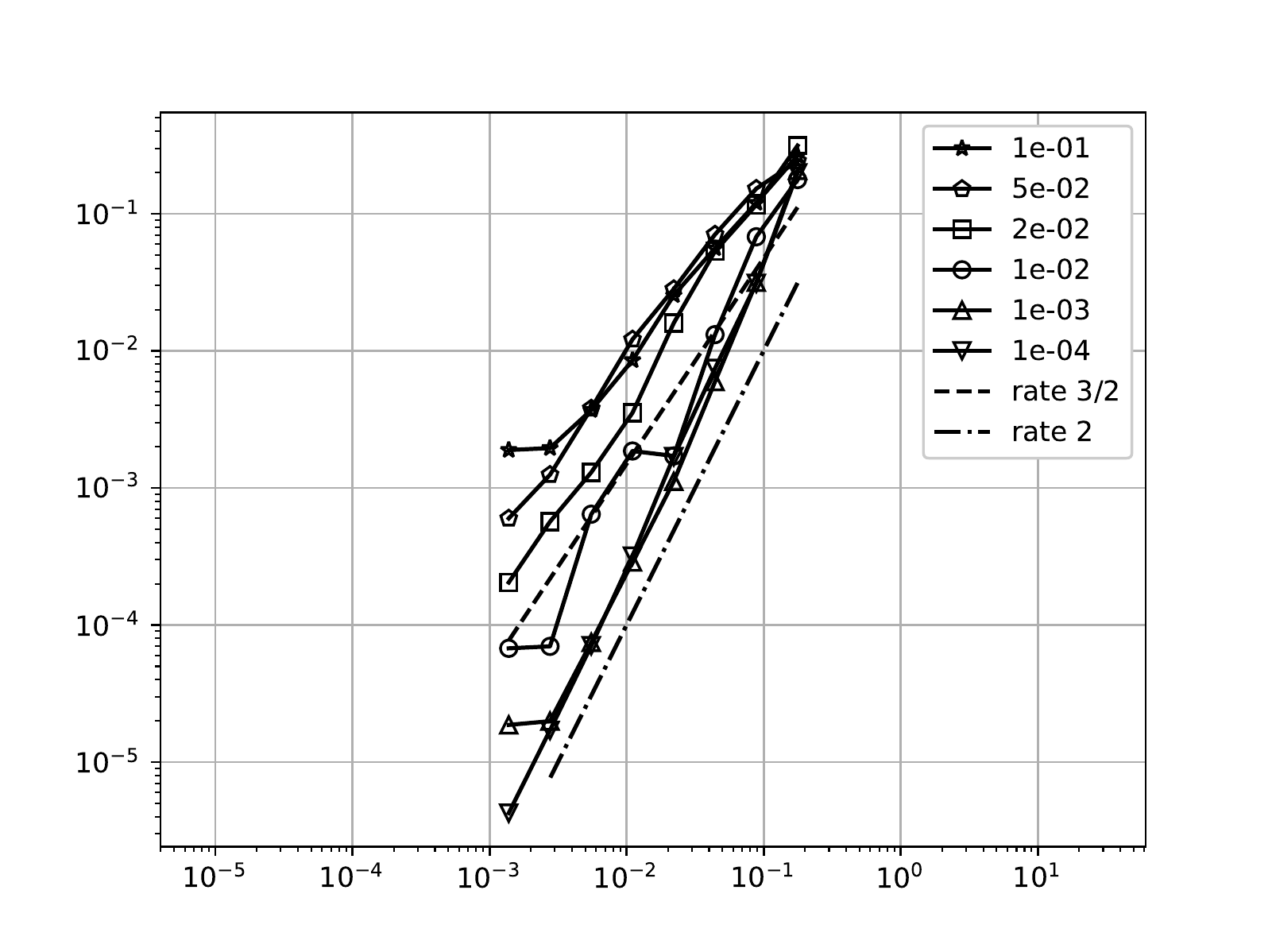}
		\caption{Upstream.}
	\end{subfigure}
	\caption{Absolute $L^2$-errors against mesh size $h$, computational domains in \cref{fig:data_sets_disk}. Varying the diffusion coefficient $\mu$ for fixed $\beta = (1,0)$, exact solution $u = 2\sin(5\pi x)\sin (5\pi y)$.}
	\label{fig:centered_disc}
\end{figure}
\subsection{Interior data set} Next we consider the setting of the example discussed in the \nameref{section:intro} (\cref{fig:contour_convection}), where data is given in the centre of the domain. We give the convergence of the $L^2$-errors in \cref{fig:centered_disc} with the caveat that this location of the data set $\omega$ is not rigorously covered by the theoretical analysis of the previous sections. Nonetheless, the experiments are in agreement with the proven results. Notice that the $L^2$-convergence is faster as $\mu$ decreases and for high P\'eclet numbers (above 10) one has optimal quadratic convergence both downstream and upstream, with the distinction that in the upstream case the convergence order is reduced in an intermediate regime, in agreement with the theoretical results in \cref{sec:error_estimates}. Also, as expected from the error estimates proven in the first part \cite{BNO20}, when diffusion is moderately small one can see the transition towards the diffusion-dominated regime as the mesh gets refined -- the convergence changes from almost quadratic to sublinear as the P\'eclet number decreases below 1.
\cref{fig:centered_disc_H1} shows almost linear convergence in the $H^1$-seminorm. We think this is observed due to the small contribution of the gradient term in the triple norm \eqref{eq:triple_norm_down}.
We also remark almost no distinction between upstream and downstream for this example, probably because the gradient term is controlled by the $L^2$-norm for small enough $\mu$.
\begin{figure}[h]
	\begin{subfigure}{0.48\textwidth}
		\includegraphics[draft=false, width=\textwidth]{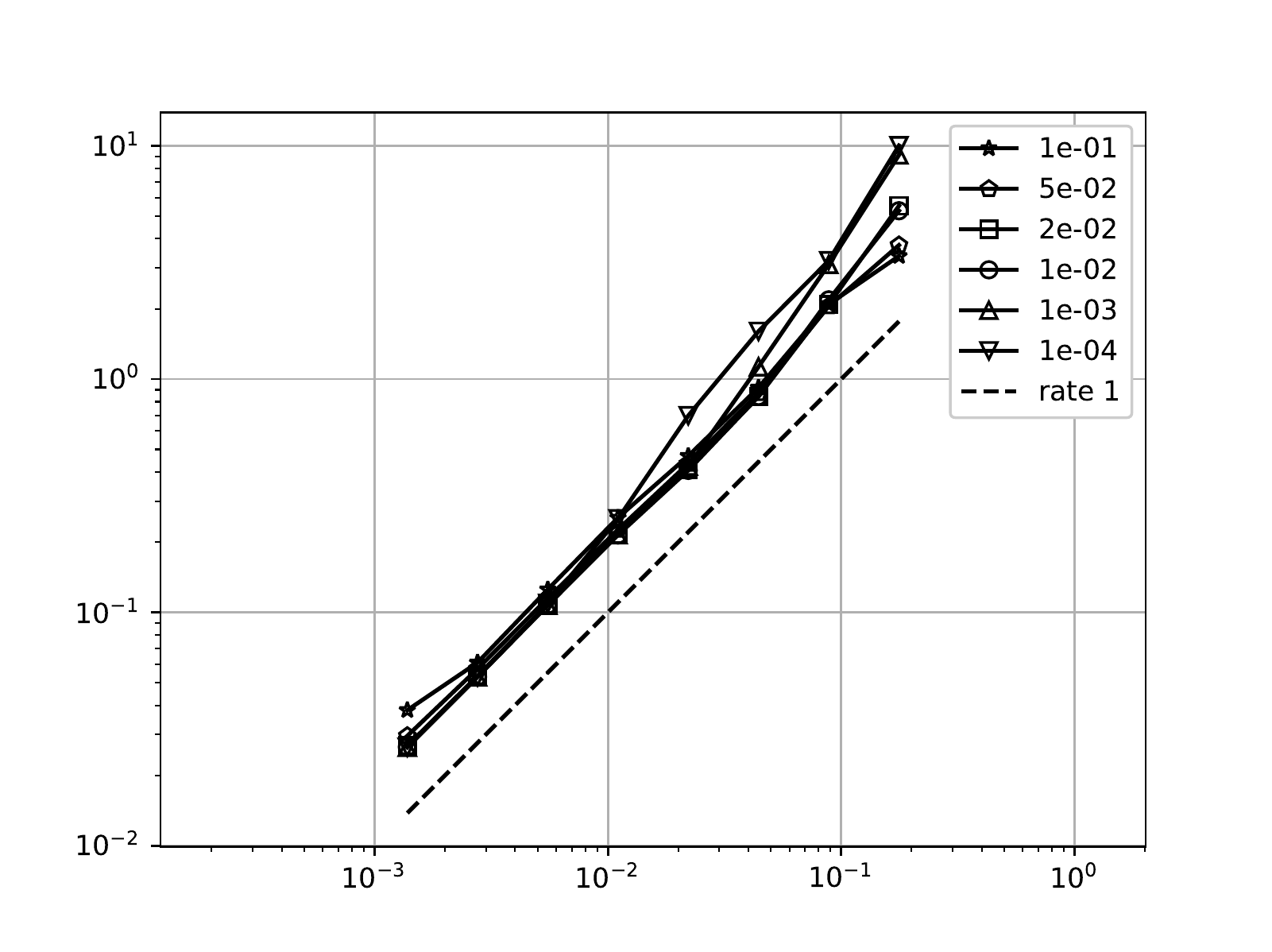}
		\caption{Downstream.}
	\end{subfigure}
	\hfill
	\begin{subfigure}{0.48\textwidth}
		\includegraphics[draft=false, width=\textwidth]{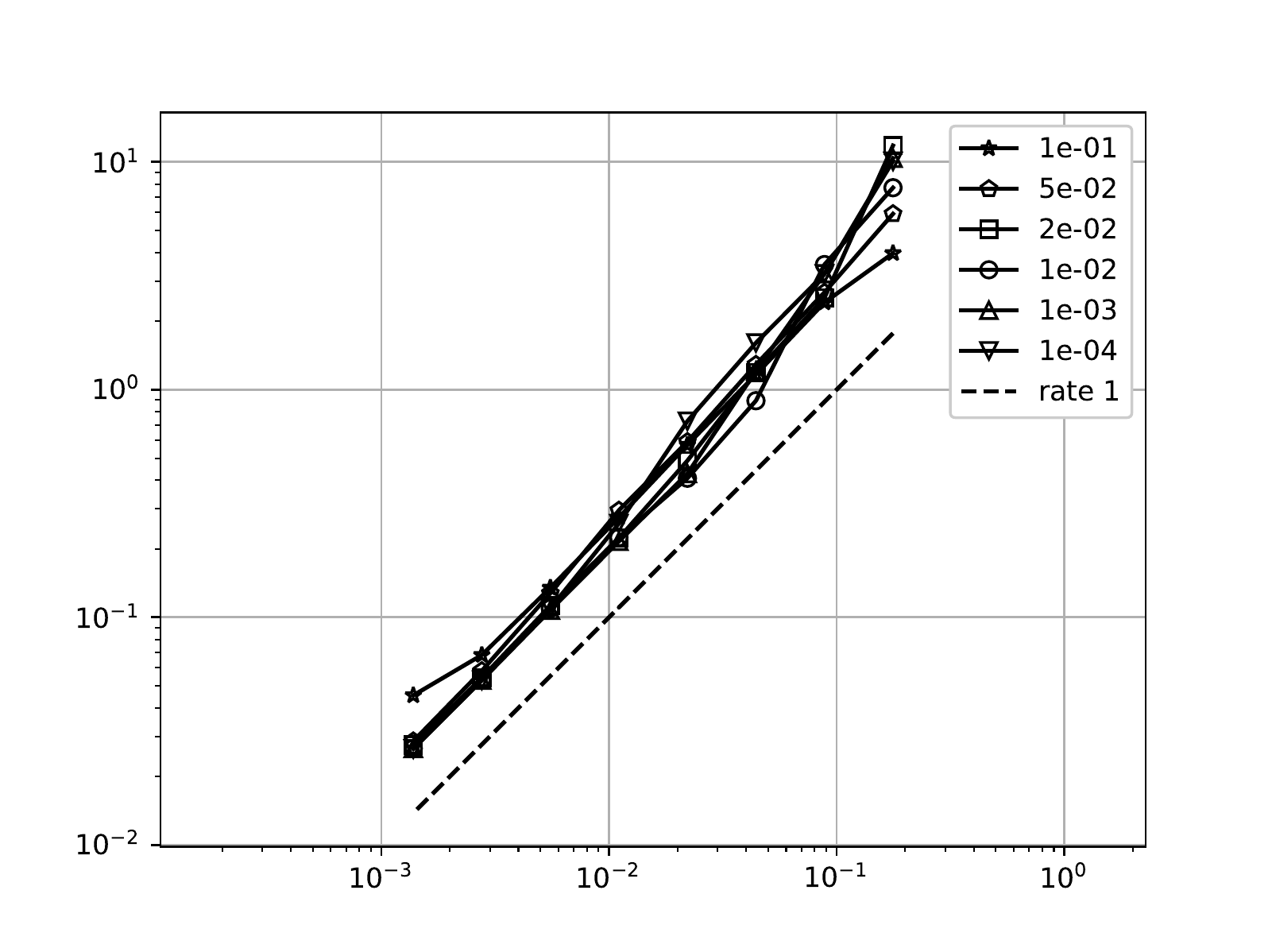}
		\caption{Upstream.}
	\end{subfigure}
	\caption{$H^1$-errors against mesh size $h$, computational domains in \cref{fig:data_sets_disk}. Varying the diffusion coefficient $\mu$ for fixed $\beta = (1,0)$, exact solution $u = 2\sin(5\pi x)\sin (5\pi y)$.}
	\label{fig:centered_disc_H1}
\end{figure}
\begin{figure}[H]
	\begin{subfigure}{0.48\textwidth}
		\includegraphics[draft=false, width=\textwidth]{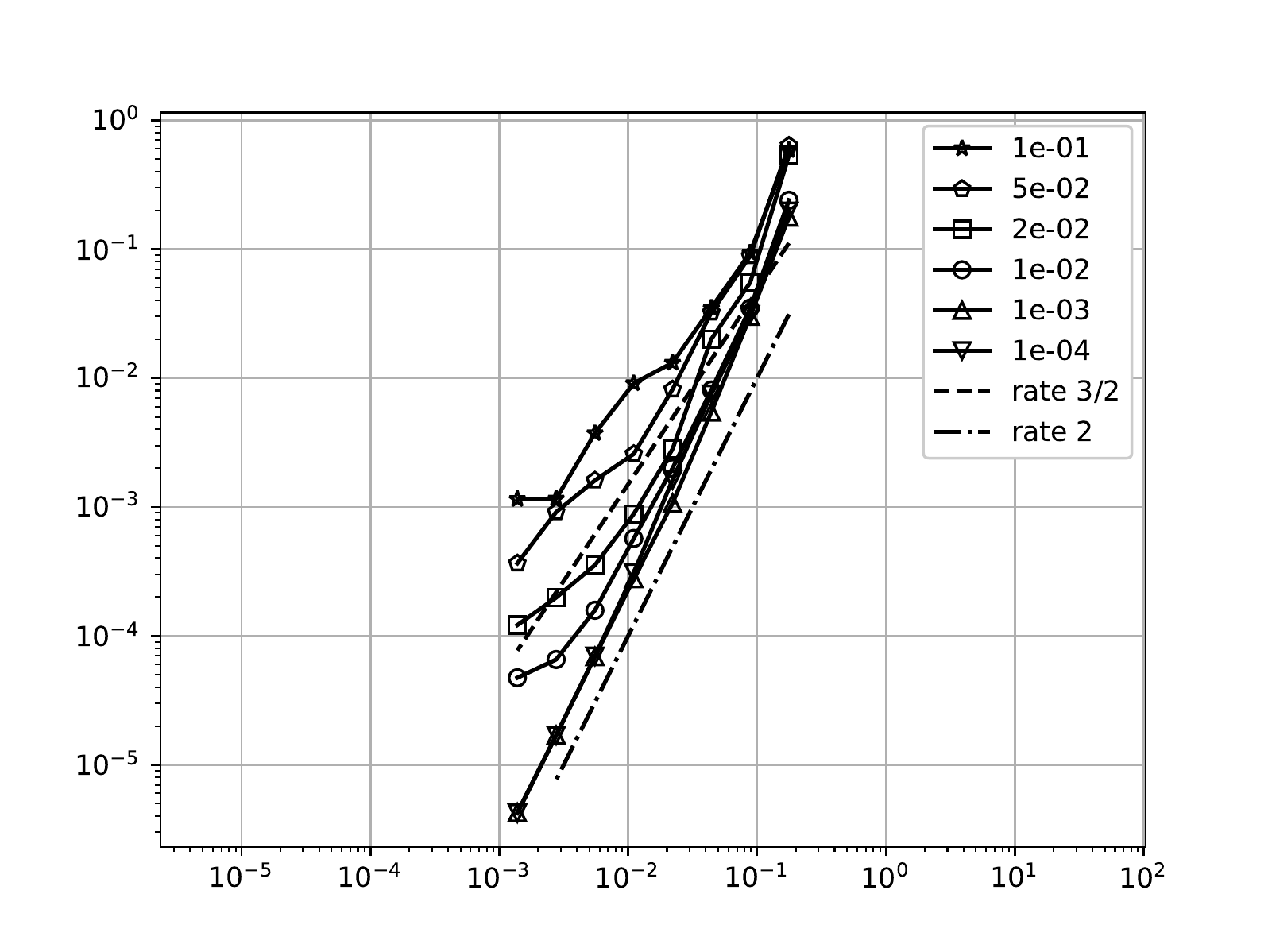}
		\caption{Downstream, perturbation $\mathcal{O}(h^2)$.}
	\end{subfigure}
	\hfill
	\begin{subfigure}{0.48\textwidth}
		\includegraphics[draft=false, width=\textwidth]{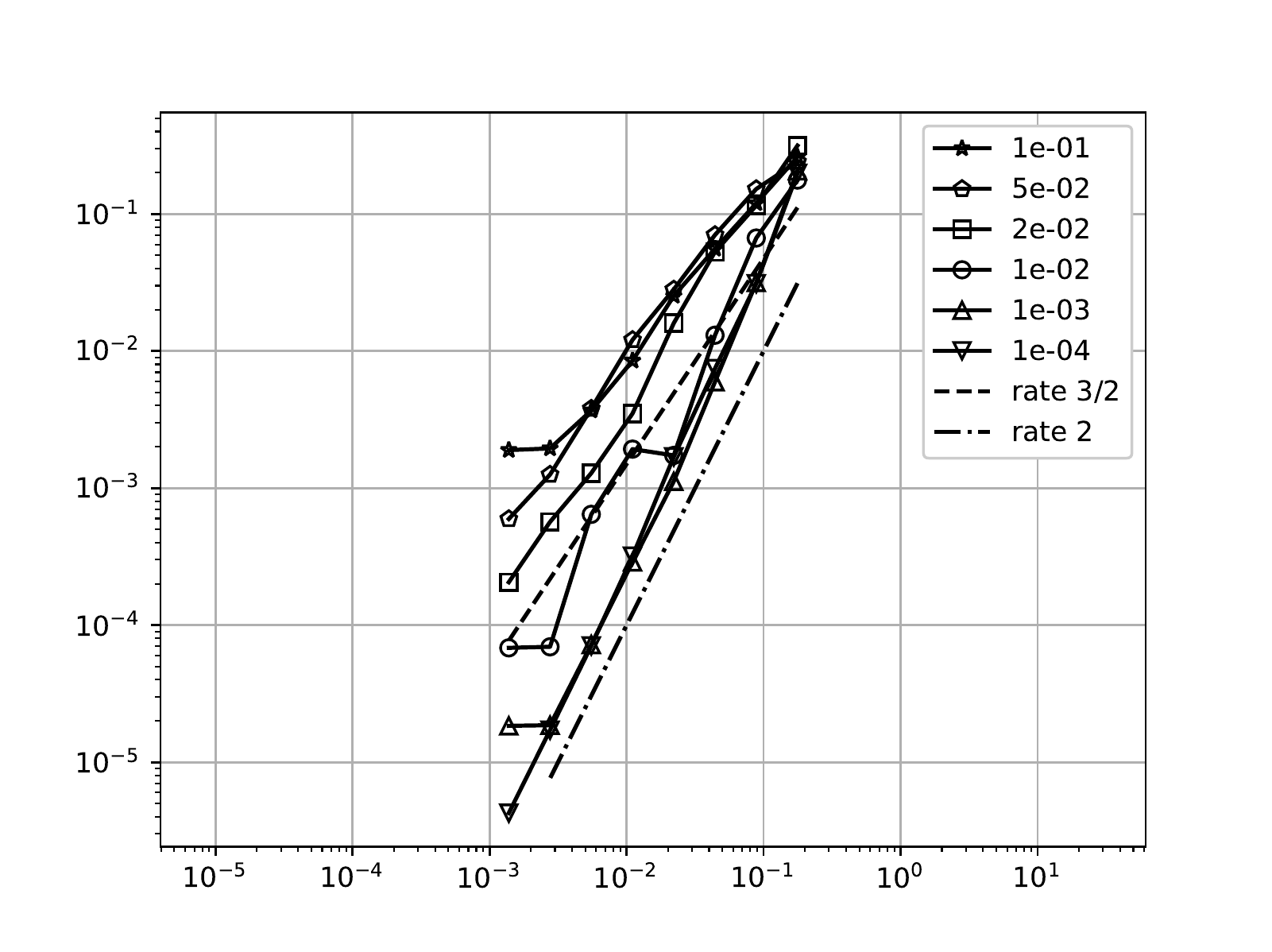}
		\caption{Upstream, perturbation $\mathcal{O}(h^2)$.}
	\end{subfigure}
	\caption{$L^2$-errors against mesh size $h$ for perturbations in data, computational domains in \cref{fig:data_sets_disk}. Varying the diffusion coefficient $\mu$ for fixed $\beta = (1,0)$, exact solution $u = 2\sin(5\pi x)\sin (5\pi y)$.}
	\label{fig:centered_disc_pert_h2}
\end{figure}
\subsection{Data perturbations}
We demonstrate the effect of data perturbations $\tilde U_\omega = u\vert_{\omega} + \delta$ in a downstream vs upstream setting by polluting the restriction of $u$ to each node of the mesh in $\omega$ with uniformly distributed values in $[-h^2,h^2], [-h, h]$ and $[-h^\frac12, h^\frac12]$, respectively.
Comparing first \cref{fig:centered_disc_pert_h2} to \cref{fig:centered_disc} we see that perturbations of amplitude $\mathcal{O}(h^2)$ have no effect on the $L^2$-errors, as expected.

An $\mathcal{O}(h)$ noise amplitude exhibits in \cref{fig:centered_disc_pert_h} the difference -- proven in \cref{thm:error_weighted_downstream} and \cref{thm:error_weighted_upstream} -- between the downstream and upstream scenarios. In the upstream case the noise has a strong effect for moderate P\'eclet numbers and the errors stagnate. Only for high P\'eclet numbers one has convergence of order $\mathcal{O}(h^\frac12)$. In the downstream case one observes lower errors, faster convergence and almost no noise effect for high P\'eclet numbers.
The difference is also very clear for perturbations of amplitude $\mathcal{O}(h^\frac12)$ shown in \cref{fig:centered_disc_pert_h12}. In the upstream case the errors stagnate and there seems to be no convergence, while in the downstream case the errors still convergence for high P\'eclet numbers.    

\begin{figure}[h]
	\begin{subfigure}{0.48\textwidth}
		\includegraphics[draft=false, width=\textwidth]{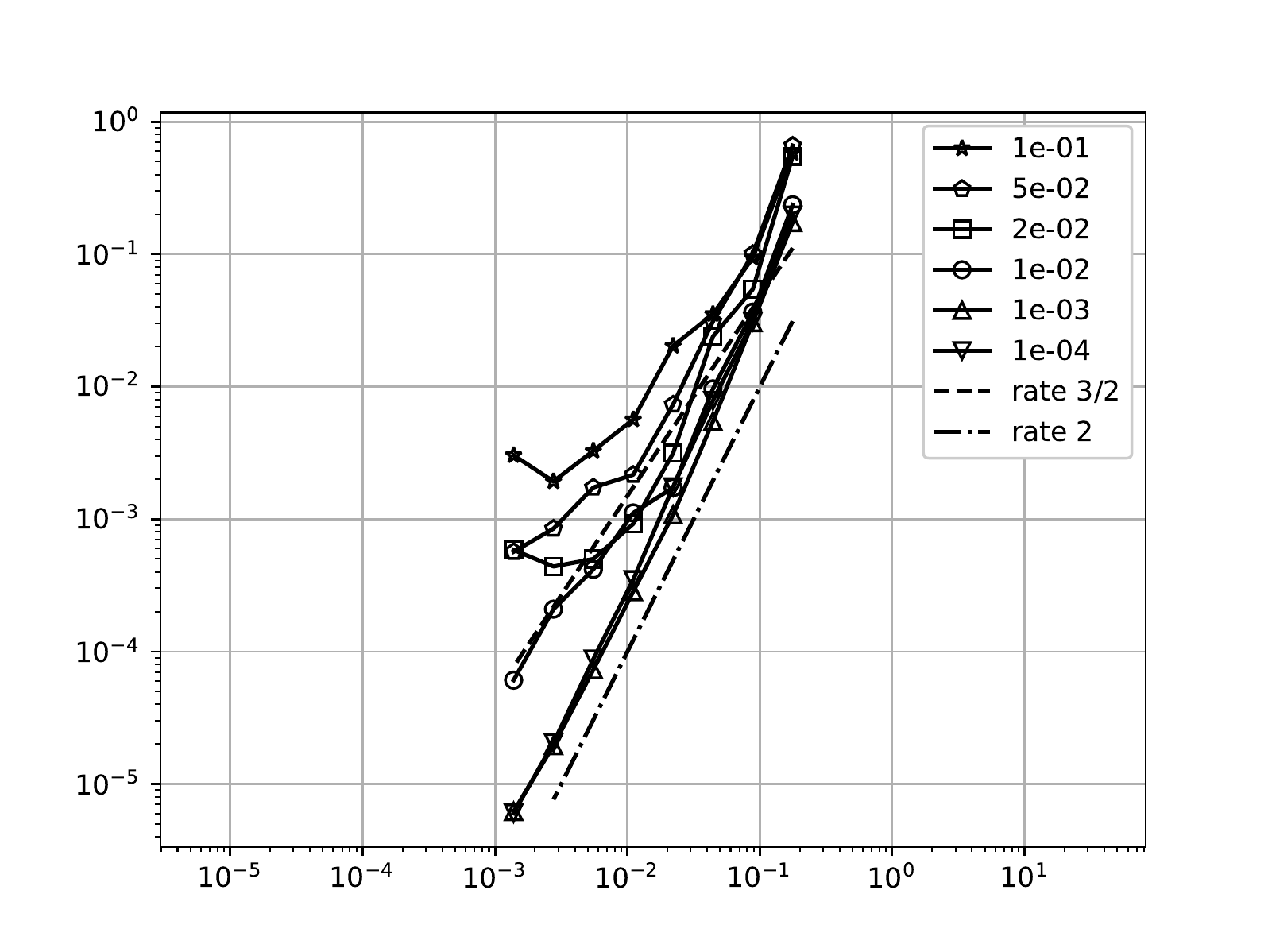}
		\caption{Downstream, perturbation $\mathcal{O}(h)$.}
	\end{subfigure}
	\hfill
	\begin{subfigure}{0.48\textwidth}
		\includegraphics[draft=false, width=\textwidth]{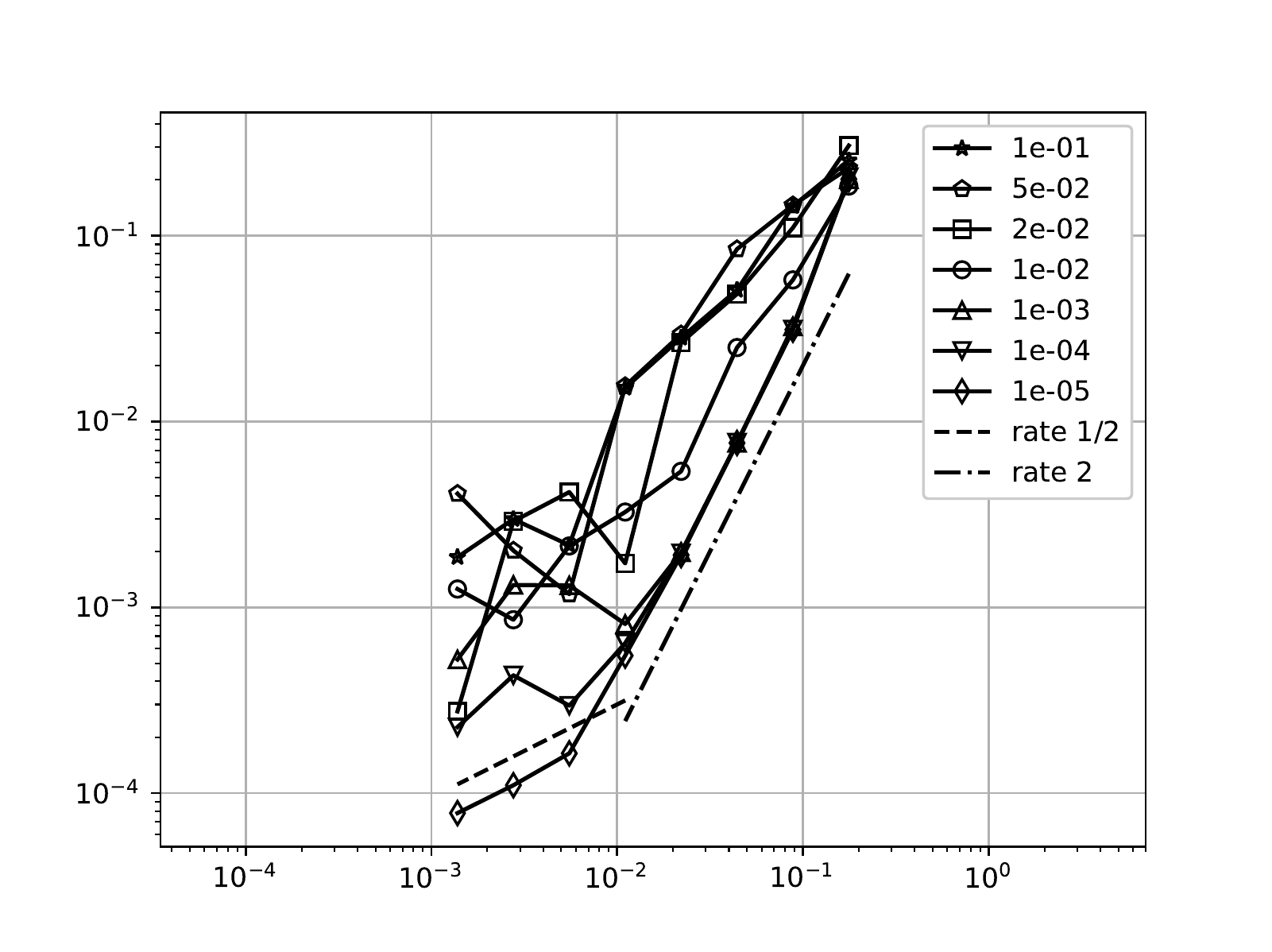}
		\caption{Upstream, perturbation $\mathcal{O}(h)$.}
	\end{subfigure}
	\caption{$L^2$-errors against mesh size $h$ for perturbations in data, computational domains in \cref{fig:data_sets_disk}. Varying the diffusion coefficient $\mu$ for fixed $\beta = (1,0)$, exact solution $u = 2\sin(5\pi x)\sin (5\pi y)$.}
	\label{fig:centered_disc_pert_h}
\end{figure}

\begin{figure}[h]
	\begin{subfigure}{0.48\textwidth}
		\includegraphics[draft=false, width=\textwidth]{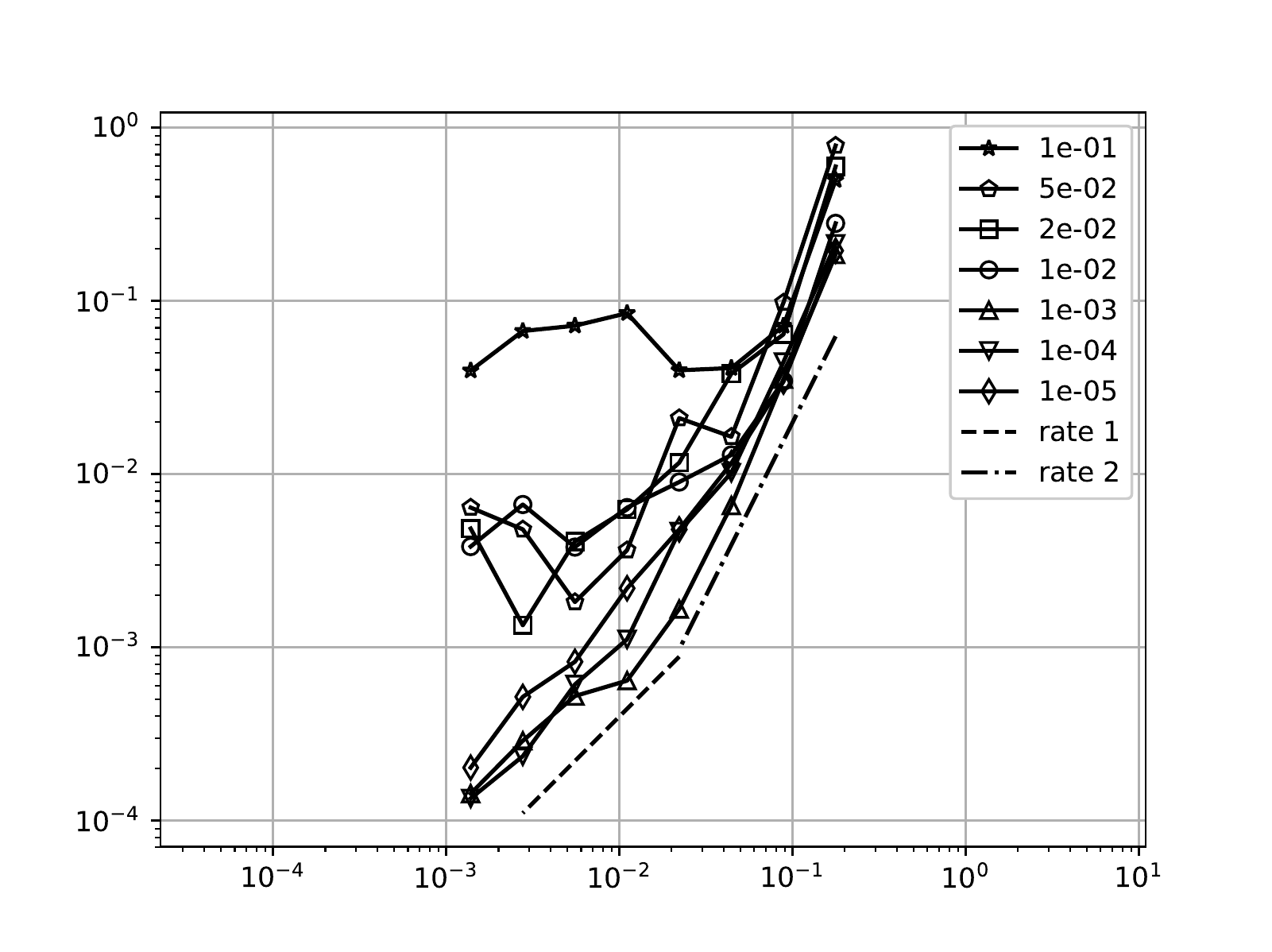}
		\caption{Downstream, perturbation $\mathcal{O}(h^\frac12)$.}
	\end{subfigure}
	\hfill
	\begin{subfigure}{0.48\textwidth}
		\includegraphics[draft=false, width=\textwidth]{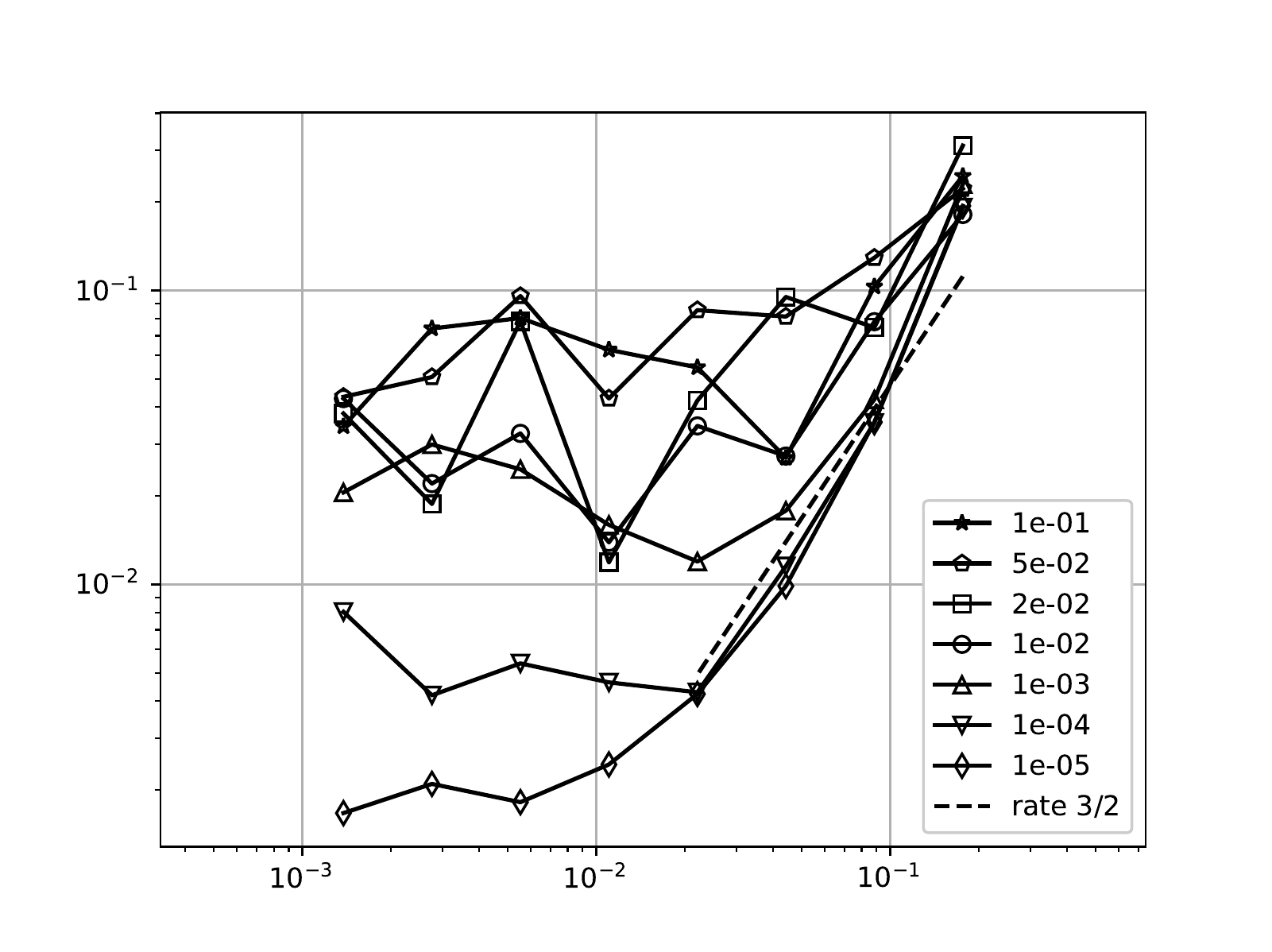}
		\caption{Upstream, perturbation $\mathcal{O}(h^\frac12)$.}
	\end{subfigure}
	\caption{$L^2$-errors against mesh size $h$ for perturbations in data, computational domains in \cref{fig:data_sets_disk}. Varying the diffusion coefficient $\mu$ for fixed $\beta = (1,0)$, exact solution $u = 2\sin(5\pi x)\sin (5\pi y)$.}
	\label{fig:centered_disc_pert_h12}
\end{figure}

\subsection{Internal layer example} We now consider an exact solution $u=\sin(3\pi x)+ \tanh(100(y-1/2))$ having an internal layer at $y=1/2$ and study qualitatively the transition from dominant diffusion to dominant convection. Data is given on both sides of the layer. The distribution of the absolute error is presented in \cref{fig:layer_diffusion} for the diffusion-dominated regime and in \cref{fig:layer_convection} for the intermediate and convection-dominated regimes.
Note that the width of the internal layer does not depend on the physical parameters.
Initially, the errors oscillate away from the data sets and concentrate around the boundary of the domain.
When convection dominates, the approximation around the layer strongly deteriorates due to the crosswind position relative to the data sets.
In this example the mesh is unstructured with 512 elements on a side and $h \approx 0.0025$. 

\begin{figure}[h]
	\includegraphics[width=0.75\columnwidth]{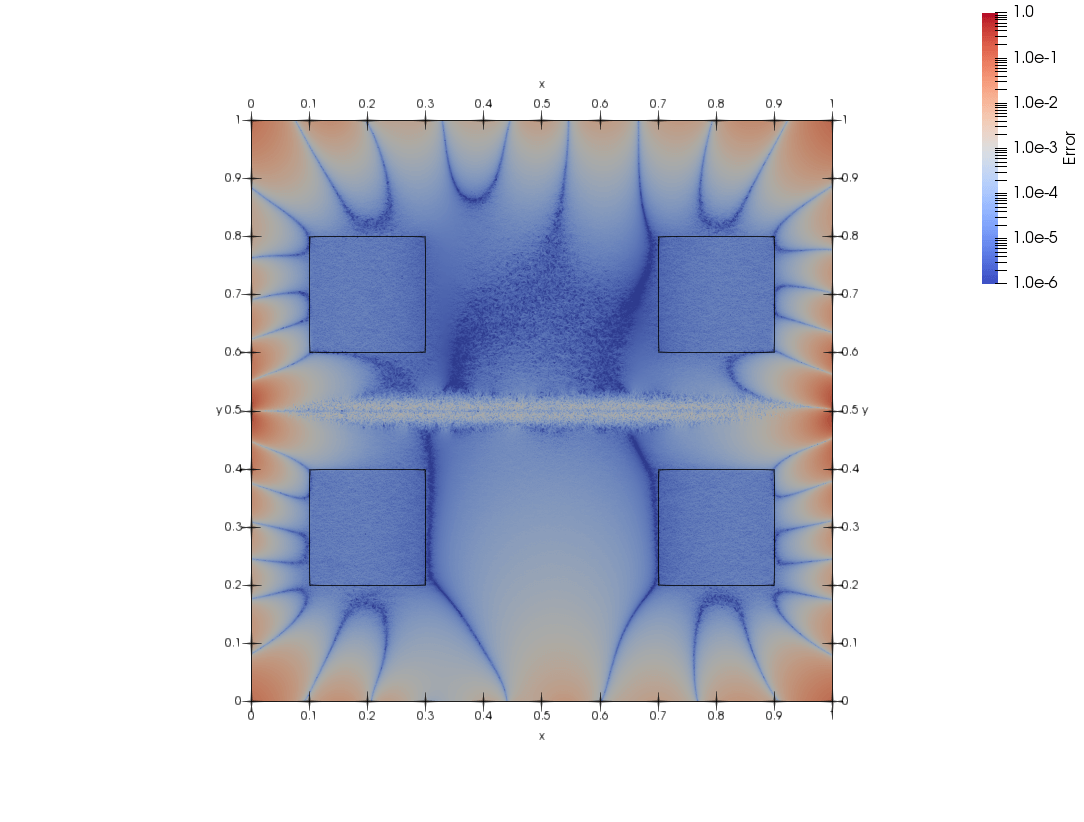}
	\caption{Absolute error in the diffusion-dominated regime, $\mu = 1,\, \beta=(1,0)$. Data given in the four outlined boxes for the solution $u=\sin(3\pi x)+ \tanh(100(y-1/2))$.}
	\label{fig:layer_diffusion}
\end{figure}

\begin{figure}[h]
	\begin{subfigure}{0.48\textwidth}
		\includegraphics[draft=false, width=\textwidth]{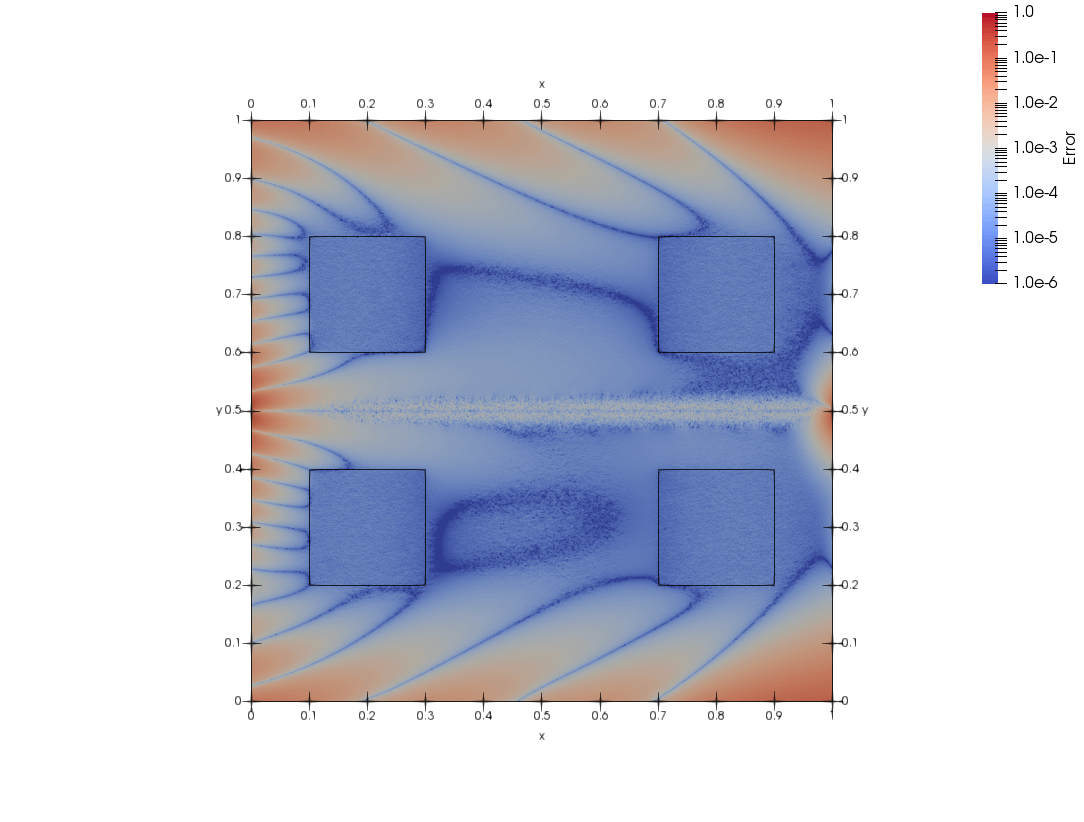}
		\caption{$\mu=10^{-2}.$}
	\end{subfigure}
	\hfill
	\begin{subfigure}{0.48\textwidth}
		\includegraphics[draft=false, width=\textwidth]{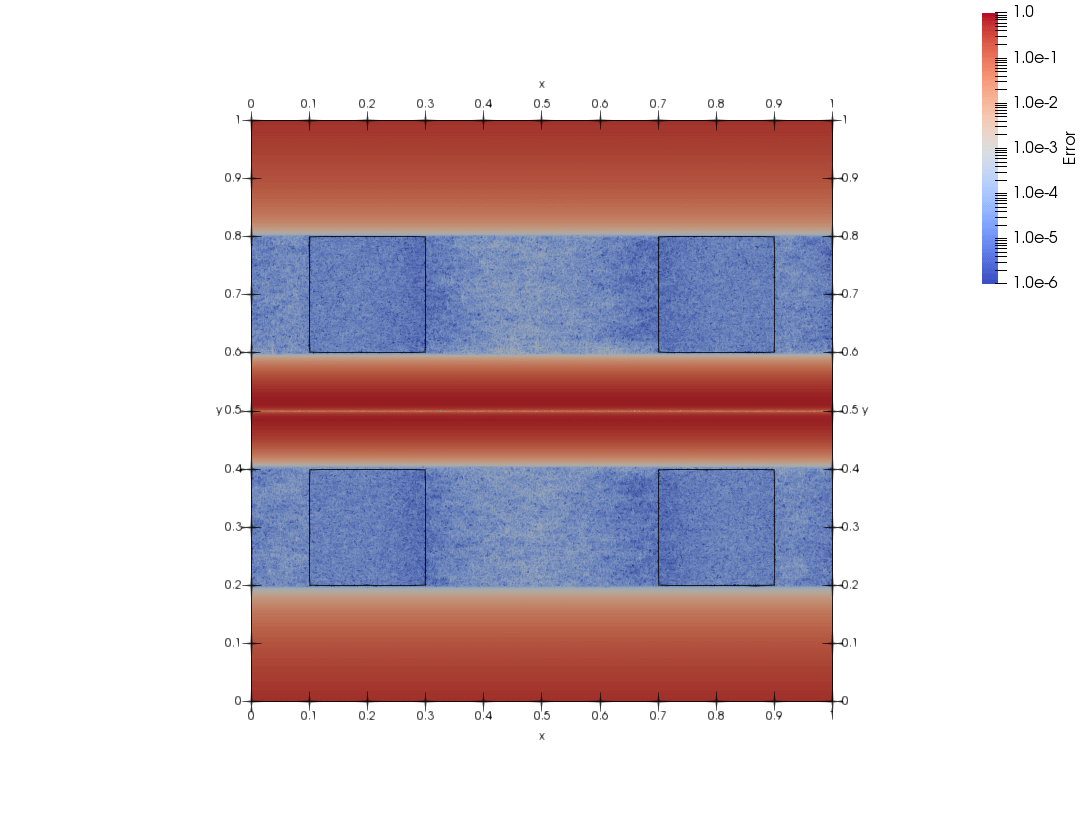}
		\caption{$\mu=10^{-6}.$}
	\end{subfigure}
	\caption{Absolute error in transition to the convection-dominated regime, $\beta=(1,0)$. Data given in the four outlined boxes for the solution $u=\sin(3\pi x)+ \tanh(100(y-1/2))$.}
	\label{fig:layer_convection}
\end{figure}

\clearpage
\bibliographystyle{alpha} 
\bibliography{biblio_part2}

\newcommand{\etalchar}[1]{$^{#1}$}
\begin{thebibliography}{ABH{\etalchar{+}}15}

\bibitem[ABH{\etalchar{+}}15]{fenics}
M.~S. Aln{\ae}s, J.~Blechta, J.~Hake, A.~Johansson, B.~Kehlet, A.~Logg,
  C.~Richardson, J.~Ring, M.~E. Rognes, and G.~N. Wells.
\newblock The {FE}ni{CS} {P}roject {V}ersion 1.5.
\newblock {\em Archive of Numerical Software}, 3(100):9--23, 2015.

\bibitem[Ber99]{Ber99}
S.~Bertoluzza.
\newblock The discrete commutator property of approximation spaces.
\newblock {\em C. R. Acad. Sci. Paris S\'er. I Math.}, 329(12):1097--1102,
  1999.

\bibitem[BGL09]{BGL09}
E.~Burman, J.~Guzm{\'a}n, and D.~Leykekhman.
\newblock Weighted error estimates of the continuous interior penalty method
  for singularly perturbed problems.
\newblock {\em IMA J. Numer. Anal.}, 29(2):284--314, 2009.

\bibitem[BHL18]{BHL18}
E.~Burman, P.~Hansbo, and M.~G. Larson.
\newblock Solving ill-posed control problems by stabilized finite element
  methods: an alternative to {T}ikhonov regularization.
\newblock {\em Inverse Problems}, 34:035004, 2018.

\bibitem[BNO20]{BNO20}
E.~Burman, M.~Nechita, and L.~Oksanen.
\newblock A stabilized finite element method for inverse problems subject to
  the convection--diffusion equation. {I}: diffusion-dominated regime.
\newblock {\em Numer. Math.}, 144(451--477), 2020.

\bibitem[Bur05]{Bur05}
E.~Burman.
\newblock A unified analysis for conforming and nonconforming stabilized finite
  element methods using interior penalty.
\newblock {\em SIAM J. Numer. Anal.}, 43(5):2012--2033, 2005.

\bibitem[Bur13]{Bur13}
E.~Burman.
\newblock Stabilized finite element methods for nonsymmetric, noncoercive, and
  ill-posed problems. {P}art {I}: {E}lliptic equations.
\newblock {\em SIAM J. Sci. Comput.}, 35(6):A2752--A2780, 2013.

\bibitem[Bur14]{Bur14a}
E.~Burman.
\newblock Stabilized finite element methods for nonsymmetric, noncoercive, and
  ill-posed problems. {P}art {II}: {H}yperbolic equations.
\newblock {\em SIAM J. Sci. Comput.}, 36(4):A1911--A1936, 2014.

\bibitem[BV07]{BV07}
R.~Becker and B.~Vexler.
\newblock Optimal control of the convection-diffusion equation using stabilized
  finite element methods.
\newblock {\em Numer. Math.}, 106(3):349--367, 2007.

\bibitem[DQ05]{DQ05}
L.~Dede' and A.~Quarteroni.
\newblock Optimal control and numerical adaptivity for advection-diffusion
  equations.
\newblock {\em M2AN Math. Model. Numer. Anal.}, 39(5):1019--1040, 2005.

\bibitem[EG04]{EG04}
A.~Ern and J.-L. Guermond.
\newblock {\em Theory and practice of finite elements}, volume 159 of {\em
  Applied Mathematical Sciences}.
\newblock Springer-Verlag, New York, 2004.

\bibitem[Eva10]{Eva10}
L.~C. Evans.
\newblock {\em Partial differential equations}, volume~19 of {\em Graduate
  Studies in Mathematics}.
\newblock American Mathematical Society, 2010.

\bibitem[HYZ09]{HYZ09}
M.~Hinze, N.~Yan, and Z.~Zhou.
\newblock Variational discretization for optimal control governed by convection
  dominated diffusion equations.
\newblock {\em J. Comput. Math.}, 27(2-3):237--253, 2009.

\bibitem[JKN18]{JKN2018}
V.~John, P.~Knobloch, and J.~Novo.
\newblock Finite elements for scalar convection-dominated equations and
  incompressible flow problems: a never ending story?
\newblock {\em Comput. Vis. Sci.}, 19(5):47--63, Dec 2018.

\bibitem[JNP84]{JNP84}
C.~Johnson, U.~N\"{a}vert, and J.~Pitk\"{a}ranta.
\newblock Finite element methods for linear hyperbolic problems.
\newblock {\em Comput. Methods Appl. Mech. Engrg.}, 45(1-3):285--312, 1984.

\bibitem[MS99]{MS99}
P.~Monk and E.~S\"{u}li.
\newblock The adaptive computation of far-field patterns by a posteriori error
  estimation of linear functionals.
\newblock {\em SIAM J. Numer. Anal.}, 36(1):251--274, 1999.

\bibitem[NS74]{NS74}
J.~A. Nitsche and A.~H. Schatz.
\newblock Interior estimates for {R}itz-{G}alerkin methods.
\newblock {\em Math. Comp.}, 28(128):937--958, 1974.

\bibitem[RS15]{RS15}
H.-G. Roos and M.~Stynes.
\newblock Some open questions in the numerical analysis of singularly perturbed
  differential equations.
\newblock {\em Comput. Methods Appl. Math.}, 15(4):531--550, 2015.

\bibitem[Ste70]{Stein70}
E.~M. Stein.
\newblock {\em Singular integrals and differentiability properties of
  functions}, volume~30 of {\em Princeton Mathematical Series}.
\newblock Princeton University Press, 1970.

\bibitem[YZ09]{YZ09}
N.~Yan and Z.~Zhou.
\newblock A priori and a posteriori error analysis of edge stabilization
  {G}alerkin method for the optimal control problem governed by
  convection-dominated diffusion equation.
\newblock {\em J. Comput. Appl. Math.}, 223(1):198--217, 2009.

\end{thebibliography}

\end{document}